\documentclass[11pt,twoside]{amsart}

\usepackage{amsmath,amssymb}
\usepackage{amsfonts}
\usepackage{amsthm,amscd}
\usepackage{amstext}
\usepackage{color}
\usepackage{latexsym}
\usepackage{graphics}
\usepackage{enumerate}

\usepackage{tikz}
\usetikzlibrary{calc}
\usetikzlibrary{decorations.pathreplacing,angles,quotes}

\usepackage{float}

\usepackage{bbm}
\usepackage{empheq}
\usepackage{mathtools}
\PassOptionsToPackage{hyphens}{url}
\usepackage{hyperref}
\hypersetup{colorlinks = false, urlcolor = blue, bookmarksopen = true}

\usepackage{geometry}
\allowdisplaybreaks

\newcommand{\cA}{\mathcal{A}}
\newcommand{\cC}{\mathcal{C}}
\newcommand{\cD}{\mathcal{D}}
\newcommand{\cM}{\mathcal{M}}
\newcommand{\cN}{\mathcal{N}}
\newcommand{\cO}{\mathcal{O}}
\newcommand{\cP}{\mathcal{P}}
\newcommand{\cS}{\mathcal{S}}
\newcommand{\cU}{\mathcal{U}}

\newcommand{\bN}{\mathbbm{N}}
\newcommand{\bZ}{\mathbbm{Z}}
\newcommand{\bR}{\mathbbm{R}}
\newcommand{\bS}{\mathbbm{S}}
\newcommand{\bT}{\mathbbm{T}}

\newcommand{\ve}{\varepsilon}
\newcommand{\intd}{\, \textrm{d}}
\newcommand{\Lip}{\textrm{Lip}}
\newcommand{\Hol}{\textrm{H\"{o}l}_{\alpha}}
\newcommand{\diam}{\mbox{diam}}
\newcommand{\hP}{\mathring{\cP}}

\newcommand{\tmax}{\tau_{\max}}
\newcommand{\tmin}{\tau_{\min}}
\newcommand{\htop}{h_{\textrm{top}}}

\newcommand{\musrb}{\mu_{\textrm{SRB}}}

\newcommand{\Bil}{\textrm{Bil}_D}
\newcommand{\BilFH}{\textrm{Bil}^{\textrm{FH}}_D}

\newcommand{\Fix}{\textrm{Fix} \,}
\newcommand{\ocap}{\textrm{ocap}}

\newtheorem{proposition}{Proposition}[section]
\newtheorem{lemma}[proposition]{Lemma}

\newtheorem{theorem}[proposition]{Theorem}

\newtheorem{corollary}[proposition]{Corollary}

\newtheorem{definition}[proposition]{Definition}

\theoremstyle{remark}
\newtheorem{remark}[proposition]{Remark}
\theoremstyle{definition}

\numberwithin{equation}{section}

\begin{document}

\title[Existence of MME]{A direct proof of the existence of MME\\ for finite horizon Sinai billiards}

\author{J{\'e}r{\^o}me Carrand}
\address{Univ. Lille, CNRS, UMR 8524 - Laboratoire Paul Painlev\'e, F-59000 Lille, France}
\email{jerome.carrand@univ-lille.fr}

\date{\today}
\begin{abstract}
The Sinai billiard map $T$ on the two-torus, i.e., the periodic Lorentz gaz, is a discontinuous map. We provide two separate conditions for a finite horizon billiard to admit equilibrium states. The first one -- namely \emph{negligible singularities} -- only concerns the billiard table and implies upper semi-continuity for the metric pressure map of $T$. 
The condition of negligible singularities is shown to be generic (dense $G_\delta$) among billiard tables with $C^{3+\alpha}$ boundaries, in the $C^{3+\alpha}$ topology, and implies the sparse recurrence condition for the null potential. 
Compactness of the set of $T$-invariant measures is proved. 
For finite horizon Sinai billiards, we provide a bound on the defect of upper semi-continuity of the metric pressure map. From it, a condition of existence of equilibrium states only involving potentials is deduced. 
In both situations, for the potentials $g \equiv 0$ and $g = -h_{\rm top}(\Phi^1) \tau$, we recover and extend the recent results of the existence of measures of maximal entropy (MME) for both the map $T$ and the flow $\Phi^t$, due to Baladi and Demers for $T$, jointly with Carrand for $\Phi^t$. 
Upper semi-continuity of the metric entropy is shown. 
Assuming negligible singularities, equidistribution of the periodic orbits with respect to the unique MME is proved. 
A bound on the topological tail entropy is given.
\end{abstract}
\thanks{The author is grateful to Viviane Baladi, Mark Demers, David Burguet and Yuri Lima for useful discussions and comments. 
Research supported by the European Research Council (ERC) under the European Union's Horizon 2020 research and innovation programme (grant agreement No 787304). The author acknowledges the support of the CDP C2EMPI, together with the French State under the France-2030 programme, the University of Lille, the Initiative of Excellence of the University of Lille, the European Metropolis of Lille for their funding and support of the R-CDP-24-004-C2EMPI project.}
\maketitle

\section{Introduction}

\subsection{Equilibrium States in Dynamical Systems}
Taking inspiration from physics, more specifically thermodynamic, where ``nature tends to maximize free energy'', it is tempting to investigate the existence of maximizing measures (or \emph{states} in physics terminology) for the \emph{pressure map}
\begin{align}\label{eq:pressure_map}
\mu \mapsto P_{\mu}(T,g) \coloneqq h_\mu(T) + \int_M g \intd\mu \, ,
\end{align}
where $\mu$ ranges over the set $\cM(M,T)$ of Borel probability measures on a space $M$ invariant under some transformation $T:M \to M$, $h_\mu(T)$ is the Kolmogorov--Sinai entropy of $(T,\mu)$, and $g : M \to \mathbbm{R}$ is some fixed function called \emph{potential} -- in the physical analogy, $\int_M -g \, \mathrm{d}\mu$ corresponds to the internal energy of the state $\mu$. Maximal measures are called equilibrium states.

This problem has been introduced by Ruelle in the case of finite rank lattices \cite{ruelleBook0,ruelleBook}, where an energized particle lies at each vertex, with finitely many possible energy values. Here, a state $\mu$ describes the energy distribution of all the particles, whereas $g$ represents the potential energy due to interactions between particles. The rank one case, assuming translation symmetry of states, corresponds to the much studied subshifts of finite type. There, the set of states is compact, and for continuous potential the pressure map is upper semi-continuous. Equilibrium states therefore exist. Assuming smoothness of the potential $g$ (e.g. H\"older), the equilibrium state $\mu_g$ is unique and enjoys strong statistical properties (exponential mixing, CLT) \cite{BowenBook}. Furthermore, $\mu_g$ can be constructed by joining the left and right eigenvectors with maximal eigenvalue of a twisted Ruelle transfer operator.

More general transformations can be studied: through the use of coding with Markov partitions, the case of Anosov or Axiom A maps $T$ can be reduced to the study of subshifts of finite type \cite{BowenBook}. In that case, for any H\"older potential, the uniquely associated equilibrium state is Bernoulli\footnote{Recall that Bernoulli implies strong mixing which implies ergodic.}. Note that the process of coding can be bypassed by studying a twisted Ruelle transfer operator acting this time on anisotropic Banach spaces \cite{baladi18book}. Although  this strategy is usually very technical, it can lead to strong statistical properties (exponential mixing, CLT, ASIP) -- provided that the operator has, for example, good spectral properties -- and has the advantage to be applicable to more general situations. 

Finally, note that if $T$ is uniformly hyperbolic on a compact manifold $M$, then it is expansive \cite[Corollary 6.4.10]{katok1997introduction}, which in turn implies that $\mu \mapsto h_\mu(T)$ is upper semi-continuous \cite[Theorem 8.2]{Walters82ergth}, giving existence for $\mu_g$ for any continuous potential $g$ since $\cM(M,T)$ is compact. Assuming further the so-called \emph{specification property} and smoothness of $g$, Bowen deduced uniqueness of the equilibrium state $\mu_g$ \cite{bowenUniqueness}, as well as sharp counting estimates on the number of periodic orbits in term of their periods \cite[Proposition~20.3.3]{katok1997introduction} and approximates the equilibrium states with long periodic orbits \cite[Theorem~20.3.7]{katok1997introduction}.
For a general continuous map $T$, the defect\footnote{By \emph{defect of upper semi-continuity} of a function $f$, we mean the difference between the upper semi-continuous envelop of $f$ and $f$.} of upper semicontinuity of the entropy map can be bounded uniformly by the topological tail entropy $h^*(T)$ \cite{misiurewicz1973diffeomorphism}. 

\subsection{The case of Sinai Billiards}

In this paper we are interested in the case where $T:M \to M$ is the collision map associated to a finite horizon Sinai billiard (see definition below). 

A Sinai billiard on the torus is (the quotient modulo $\mathbbm{Z}^2$, for position, of) the periodic planar
Lorentz gas (1905) model for the motion of a single dilute electron in a metal. The scatterers
(corresponding to atoms of the metal) are assumed to be strictly convex (but not necessarily
discs) and fixed. Such billiards have become fundamental models in mathematical physics. 

Assuming finite horizon (see definition below), many equilibrium states for the billiard map (resp. for the billiard flow) have already been studied, such as the SRB measure \cite{chernov2006chaotic} which enjoys exponential mixing \cite{Y,DZ1} (resp. \cite{BDL2018ExpDecay}), the measure of maximal entropy, constructed in \cite{BD2020MME} (resp. \cite{Ca1, BCD}) enjoys polynomial mixing \cite{DK}, equilibrium states associated to piecewise-H\"older potentials \cite{Ca1, BCD} which are Bernoulli, and finally equilibrium states with geometric potentials (which are not H\"older) \cite{BD2} which are exponentially mixing. Most of the constructions -- those of \cite{BD2020MME,Ca1,BCD,BD2} -- or the study of the equilibrium states -- as for \cite{DZ1, BDL2018ExpDecay} -- rely on the use of a transfer operator acting on specific anisotropic Banach spaces. In \cite{BCD,BD2020MME,BD2,Ca1,LOP24}, the studied equilibrium states are proved to be unique. In the construction performed in \cite{BD2020MME, Ca1, BCD}, an additional assumption on the billiard table is made, namely \emph{sparse recurrence to singularities} (with respect to $g$), that is 
\begin{align}\label{eq:sparse_recurrence}
P_{\rm top}(T,g) - \sup g > \inf_{n_0,\varphi_0} s_0(n_0,\varphi_0) \log 2,
\end{align}
where $s_0 \in (0,1)$ is related to the return frequency near $\partial M$ (see \eqref{eq: def s_0}) and $P_{\rm top}(T,g)$ is the topological pressure of $g$ under the transformation $T$ (see \eqref{eq: def Ptop}).

In this paper, we say that a billiard table has \emph{negligible singularities} when $\partial M$ is a null set with respect to every invariant measure. It turned out that this condition is equivalent to the vanishing of $\inf s_0$ (in particular, it implies sparse recurrence for $g\equiv 0$), and is generic in the sense it holds on a dense $G_\delta$ set.

\medskip

The first half of this paper is dedicated to a generic class of billiard tables with finite horizon. In Section~\ref{sect:bounded_complexity}, we show that for any finite horizon Sinai billiard with negligible singularities and any piecewise continuous potential, the metric pressure map \eqref{eq:pressure_map} is upper semi-continuous. In Section~\ref{sect:closed}, we prove that for any finite horizon Sinai billiard, $\cM(M,T)$ is compact. In particular for $g \equiv 0$ and $g = -h_{\rm top}(\Phi^1) \tau$ (see \cite{Ca1}), we recover the existence of the MME for the collision map and the billiard flow respectively. Furthermore, we prove that the variational principle is satisfied. In Section~\ref{sect: typicality of NS} we prove that the negligible singularities condition is generic with respect to the $C^{3+\alpha}$ topology. Finally, in Section~\ref{sect: NS and s0} we prove that the singularities are negligible if and only if $\inf s_0$ vanishes.

In the second half of the paper, we study general billiard table with finite horizon. 
In Section~\ref{sect:regularity_and_consequences}, we provide a bound on the defect of upper semi-continuity of the pressure map, similar to the ones obtained in \cite{ITV22,EKP15,KP17,burguet2024,velozo2017} in other contexts, and derive consequences: existence of equilibrium states (assuming $ P_{\rm top}(T,g) > \sup g $, which holds for both $g \equiv 0$ and $g = - \htop(\Phi^1)\tau$) and equidistribution of the periodic orbits with respect to the MME (assuming negligible singularities). In Section~\ref{sect:tail_entropy} we provide an upper bound on the topological tail entropy of the billiard map. Finally, in Appendix~\ref{sect: examples}, we give some examples of billiard tables satisfying both sparse recurrence and negligible singularities. 

The problem of uniqueness is not covered by this paper. In~\cite{LOP24}, it is shown -- without requiring sparse recurrence -- that for any fixed piecewise H\"older potential, among all associated equilibrium states, at most one is adapted.

\subsection{Precise statement of results}\label{subsect: precise statements}

Formally, a Sinai billiard table $Q$ on the two-torus $\mathbbm{T}^2$ is a set $Q = \mathbbm{T}^2 \smallsetminus B$, with $B = \bigcup_{i=1}^D B_i$ for some finite number $D \geqslant 1$ of pairwise disjoint closed domains $B_i$, called scatterers, with $\mathcal{C}^3$ boundaries having strictly positive curvature -- in particular, the scatterers are strictly convex. The billiard flow $\Phi^t : \Omega \to \Omega$ is the motion of a point particle travelling at unit speed in $Q$ with specular reflections off the boundary of the scatterers. Identifying outgoing collisions with incoming ones in the phase space, that is $\Omega = Q \times \mathbbm{S}^1 / \sim$, the billiard flow is continuous. However, the grazing collisions -- those tangential to scatterers -- give rise to singularities in the derivative~\cite{chernov2006chaotic}. The Sinai billiard map $T : M \to M$ --~also called collision map~-- is the first return map of the single point particle to the scatterers, that is $M = \partial Q \times [-\pi /2, \pi /2]$. Because of the grazing collisions, the Sinai billiard map is a discontinuous map.

As in~\cite{BD2020MME,Ca1,BCD}, we say that the billiard table $Q$ has finite horizon if the billiard flow does not have any trajectory making only tangential collisions -- in particular, this implies that the return time function $\tau$ to the scatterers is bounded. All Sinai billiards in this paper are assumed to have finite horizon. In this case, $\Phi^t$ is a suspension flow over $T$, under $\tau$.

Define $\cS_0 = \{ (r, \varphi) \in M \mid | \varphi | = \pi/2 \}$ the set of grazing collisions, $\cS_{\pm n} = \bigcup_{i=0}^n T^{\mp i} \cS_0 $ the singular set of $T^{\pm n}$, and $\cS = \bigcup_{i \in \bZ} \cS_i$ the set of all singularities. Notice that $\cS_{n}$ (resp. $\cS_{-n}$) is made of a finite number of decreasing (resp. increasing) curves in the $(r,\varphi)$ coordinates, whereas $M$ is a two dimensional manifold with boundary $\partial M = \cS_0$. The return frequency of an orbit near $\cS_0$ can be quantified by
\begin{align}\label{eq: def s_0}
s_0(n_0,\varphi_0) \coloneqq \frac{1}{n_0} \sup_{M} \sum_{i=0}^{n_0-1} \mathbbm{1}_{ ( |\varphi| > \varphi_0 ) } \circ T^i,
\end{align}
where $n_0 \geqslant 1$ and $0 < \varphi_0 < \pi/2$. This is the maximal frequency of $\varphi_0$-grazing collisions among $n_0$ consecutive collisions. This quantity is crucial in the work of Baladi and Demers~\cite{BD2020MME}.

The complexity of a billiard is a sequence $(K_n)_{n\geqslant 1}$, where each $K_n$ is defined as the maximal number of curves of $\cS_n$ intersecting at some common point. Chernov~\cite{chernov01} proved that for any finite horizon Sinai billiard, there exists a constant $K$ such that $K_n \leqslant K n$ -- note that linear growth can be achieved if there exists a grazing periodic orbit. It is conjectured that bounded complexity holds generically among finite horizon Sinai billiards. T\'oth announced in a recent seminar~\cite{toth} the resolution of the conjecture, and a paper is in preparation.

In~\cite{DK}, Demers and Korepanov related the complexity sequence to the $s_0$ parameter. They proved\footnote{Actually, they proved $\inf_{n_0,\varphi_0} s_0(n_0,\varphi_0) \leqslant \liminf_n K_n /n $.} that if the complexity is bounded, then $\inf s_0 = 0$. In particular, bounded complexity implies sparse recurrence for $g \equiv 0$.

The topological pressure is defined using a particular familly of partitions of $M$. Define $\cP$ the partition of $M$ into maximal connected components on which both $T$ and $T^{-1}$ are continuous, and $\cP_{-k}^n = \bigvee_{i=-k}^n T^{-i} \cP$. As in~\cite{Ca1}, for any bounded potential $g$, the topological pressure of $T$ and $g$ is defined as 
\begin{align}\label{eq: def Ptop}
P_{\rm top}(T,g) \coloneqq \lim_{n \to \infty} \frac{1}{n} \log \sum_{A \in \cP_0^n} \sup_A e^{S_n g} ,
\end{align}
where $S_n g$ is the $n$-th Birkhoff sum $\sum_{k=0}^{n-1} g \circ T^k$ of $g$. 

Let $\cM_{-k}^n$ be the partition of $M \smallsetminus ( \cS_{-k} \cup \cS_n)$ into maximal connected components (which are all open sets of $M$). We say that a potential $g$ is $\cM_{-k}^n$-continuous if $g$ is bounded and $g|_A$ is continuous for each $A \in \cM_{-k}^n$. We simply say that a potential is piecewise continuous if it is $\cM_{-k}^n$-continuous for some $k,n \geqslant 0$.

In this paper, we introduce an intermediate\footnote{in the sense: bounded complexity $\Rightarrow$ negligible singularities $\Rightarrow$ sparse recurrence for $g \equiv 0$.} condition, namely \emph{Negligible Singularities}. A billiard table satisfies this condition whenever $\mu(\cS)=0$ for every $T$-invariant measure $\mu$. The present work does not rely on the bounded complexity conjecture.

The first half of the present paper -- that is, Sections~\ref{sect:bounded_complexity},~\ref{sect:closed} and~\ref{sect: typicality of NS} -- is devoted to this condition, where we prove the following result.
\begin{theorem}\label{thm:main_existence_generic}
For any finite horizon Sinai billiard, $\cM(M,T)$ is compact with respect to the weak-$^*$ topology. If negligible singularities holds, then for any piecewise-continuous potential $g : M \to \mathbbm{R}$, the metric pressure map $\mu \in \cM(M,T) \mapsto P_{\mu}(T,g)$ is upper semi-continuous. Thus equilibrium states for $g$ exist, and moreover, they all satisfy the variational principle.\\
In particular, $g \equiv 0$ and $g = - h_{\rm top}(\phi_1)\tau$ admit equilibrium measures. In other words, $T$ and $\Phi^t$ have MME.\\
Furthermore, finite horizon Sinai billiard table with negligible singularities form a dense $G_\delta$ set within the set of finite horizon tables with respect to the $C^{3+\alpha}$ topology.
\end{theorem}
The upper semi-continuity statement in the above theorem follows from Proposition~\ref{prop:USC of KS entropy}. The compactness of $\cM(M,T)$ is proved in Proposition~\ref{prop: closed set}. The relation between $g = - h_{\rm top}(\phi_1)\tau$ and the MME of the Sinai billiard flow is proved in~\cite[Corollary~2.6]{Ca1}. The genericness is proved in Proposition~\ref{prop: U_k are open and dense}.

\medskip

In order to state our second result, we recall that the topological tail entropy $h^*(T) \geqslant 0$ quantifies the complexity of a dynamical system at arbitrarily small scales. In a sense, it is the maximum over every orbit of entropy created near this single orbit. A precise definition is given in Section~\ref{sect:tail_entropy}.

The second half of the paper, namely Sections~\ref{sect: NS and s0}, \ref{sect:regularity_and_consequences} and~\ref{sect:tail_entropy}, is devoted to the proof of the following result.
\begin{theorem}\label{thm:main_2}
For any finite horizon Sinai billiard table:

\noindent i) Singularities are negligible if and only if $\inf_{(n_0,\varphi_0)} s_0(n_0,\varphi_0) =0$;

\noindent ii) $\limsup_{\nu \rightharpoonup \mu} P_{\nu}(T,g) \leqslant P_{\mu}(T,g) + \mu(\cS) \left( \min \{ P_{\rm top}(T,g), \sup g \} - P_{\mu_\cS}(T,g) \right) $, for every $\cM_0^1$-continuous potential $g$ and every $\mu \in \cM(M,T)$, where $\mu_\cS$ is the conditional measure given by $\mu(\cS) \mu_\cS(A) = \mu(A \cap \cS)$;

\noindent iii) $h^*(T) \leqslant \left( 3+ 2\left\lfloor \frac{\max \tau}{\min \tau} \right\rfloor \right) s_0 \log 2 K$, where $K$ is such that $K_n \leqslant Kn$ for every $n \geqslant 1$.

In particular, if the singularities are negligible then ii) implies the upper semi-continuity of the pressure map (and of the entropy map in the case $g \equiv 0$), and iii) implies that $h^*(T)=0$.
\end{theorem}
The above theorem will follow from Proposition~\ref{prop: NS and zero s0} for \textit{i)}, from Corollary~\ref{prop:defect of USC} and Lemma~\ref{lemma: USC-ish} for \textit{ii)} and from Corollary~\ref{prop: top tail entropy bound} for \textit{iii)}. From \textit{ii)}, we deduce the following

\begin{corollary}
For any finite horizon Sinai billiard table:

\noindent i) For any $\cM_0^1$-continuous potential $g$ such that $P_{\rm top}(T,g) - \sup g > 0 $, there exist equilibrium states for $g$. In particular, for $g \equiv 0$ and $g=-h_{\rm top}(\Phi^1) \tau$, both billiard map and flow admit measures of maximal entropy;

\noindent ii) The Kolmogorov--Sinai entropy of the billiard map is upper semi-continuous;

\noindent iii) If the singularities are negligible, then sparse recurrence for $g \equiv 0$ holds too and periodic orbits equidistribute with respect to the unique measure of maximal entropy.
\end{corollary}

The corollary follows from Corollary~\ref{corol: existence FH} for \textit{i)}, from Theorem~\ref{thm: USC KS entropy} for \textit{ii)} and from Corollary~\ref{corol: equidist of per orbits} for \textit{iii)}.

In Appendix~\ref{sect: examples}, we consider one-parameter families of Sinai billiard tables. We compute a lower bound on their topological entropy and deduce a set of parameters for which sparse recurrence (for $g\equiv 0$) holds, and provide generic sets of parameter for which the singularities are negligible.

\begin{remark}
During the editorial process, other similar results appeared in the literature, partially overlapping with those presented in this article. Finamore and Leguil~\cite{FL2025cat} proved the genericity of the negligible singularities condition with respect to the $C^3$ topology, as part of an article concerning the marked length spectrum rigidity for Sinai billiards. Climenhaga and Day~\cite{CD26} used an involved coding approach to prove both existence and uniqueness of the MME of the Sinai billiard map under the only assumption of finite horizon.
\end{remark}

\section{Negligible Singularities case}\label{sect:bounded_complexity}

In this section we prove the upper semi-continuity of \eqref{eq:pressure_map} under negligible singularities assumption. We begin with an intermediate result. Since $\cP$ -- the coarsest partition of $M$ on whose atoms $T$ and $T^{-1}$ are continuous -- is a generating partition, we begin by proving the continuity of each $\mu \mapsto \frac{1}{n} H_\mu(\cP_l^n)$ in Lemma~\ref{prop:continuity static entropy}, where 
\begin{align*}
H_\mu(\cA) \coloneqq - \sum_{A \in \cA} \mu(A) \log \mu(A)
\end{align*}
denotes the static entropy of a partition $\cA$. For this, we will use that when the singularities are negligible, $\mu(\partial A)=0$ for every $A \in \cP_l^n$, all $n,l \geqslant 0$ and every $\mu \in \cM(M,T)$. We deduce the upper semi-continuity of \eqref{eq:pressure_map} in Proposition~\ref{prop:USC of KS entropy}.

\begin{lemma}\label{prop:continuity static entropy}
For any finite horizon Sinai billiard with negligible singularities and for any $n,\ell \geqslant 0$, the map $\mu \mapsto \frac{1}{n} H_\mu(\cP_{-\ell}^n)$ is continuous on $\cM(M,T)$.
\end{lemma}

\begin{proof}
Fix some $n,\ell \geqslant 0$. We start by proving that for every $A \in \cP_{-\ell}^n$, $\partial A$ has measure zero. Notice that for every $A \in \cP_{-\ell}^n$, $\partial A \subset \cS_{-\ell-1} \cup \cS_{n+1} \subset \cS$  (see~\cite[Lemma~3.3]{BD2020MME}).

Let $\mu_k \rightharpoonup \mu$ in $\cM(M,T)$ as $k \to \infty$. Let $A \in \cP_{-\ell}^n$. By the above, $\mu(\partial A)=0$. Applying Portmanteau theorem, we thus get $\lim_{k \to \infty} \mu_k(A) = \mu(A)$. Since $\cP_{-\ell}^n$ is a finite partition, and by continuity of $x \mapsto -x \log x$ on $[0,1]$, we get the desired result.  
\end{proof}

We can now prove the upper semi-continuity of \eqref{eq:pressure_map}.

\begin{proposition}\label{prop:USC of KS entropy}
Assuming negligible singularities, then for any $\cM_{-\ell}^n$-piecewise continuous potential $g : M \to \mathbbm{R}$ the map $\mu \mapsto h_\mu(T) + \int g \,\mathrm{d}\mu$ is upper semi-continuous on $\cM(M,T)$.
\end{proposition}

\begin{proof}
First, we prove that $\mu \mapsto h_\mu(T)$ is upper semi-continuous. Let $\mu \in \cM(M,T)$, $\ve >0$. By the continuity from Lemma~\ref{prop:continuity static entropy}, we get that for every $n \geqslant 1$, there is a neighbourhood $U_{n,\ve}$ of $\mu$ such that for every $\nu \in U_{n,\ve}$, $\frac{1}{n}H_\mu(\cP_0^n) \geqslant \frac{1}{n} H_\nu(\cP_0^n) -\ve$. Therefore, since $\cP$ is a generating partition,
\begin{align*}
h_\mu(T) &= h_\mu(T, \cP) = \lim_{n \to \infty} \frac{1}{n} H_{\mu}(\cP_0^n) \geqslant \lim_{n \to \infty} \sup_{\nu \in U_{n,\ve}} \frac{1}{n} H_\nu(\cP_0^n) - \ve \geqslant \lim_{n \to \infty} \sup_{\nu \in U_{n,\ve}} h_\nu(T,\cP) - \ve \\
&\geqslant \lim_{n \to \infty} \sup_{\nu \in U_{n,\ve}} h_\nu(T) - \ve = \limsup_{\nu \rightharpoonup \mu} h_\nu(T) - \ve,
\end{align*}
where for the second to last inequality we used that $h_\nu(T,\cP) = \inf_{n} \frac{1}{n}H_\nu(\cP_0^n)$. 
Since this holds for any $\ve >0$, we get the upper semi-continuity of the entropy.

\medskip
We now prove that $\mu \mapsto \int g \,\mathrm{d}\mu$ is continuous. Let $g$ be a $\cM_{-\ell}^n$-continuous potential and $\mu \in \cM(M,T)$. Since \[ \int g \,\mathrm{d}\mu = \sum_{A \in \cM_{-\ell}^n} \int g \mathbbm{1}_A \intd \mu, \] we only need to prove the continuity of $\mu \mapsto \int g \mathbbm{1}_A \, \mathrm{d}\mu$ for any fixed $A \in \cM_{-\ell}^n$. Notice that, since $\partial A$ is compact, we have $\partial A = \bigcap_{\delta > 0} \cN_\delta(\partial A)$, where $\cN_\delta(\partial A)$ is a $\delta$-neighbourhood of $\partial A$. Let $\ve > 0$. Then, since $\mu (\partial A) = 0$, there exists some $\delta > 0$ such that $\mu( \cN_\delta(\partial A) ) < \ve$. Up to decreasing the value of $\delta$, we can also assume that $\mu (\partial \cN_\delta(\partial A) ) = 0$. Let $\mu_k \rightharpoonup \mu$ as $k \to \infty$. Define 
\begin{align}\label{eq:def phi_A}
\varphi_{A,\delta}(x) = \begin{cases} 1 & \text{if } x \in A \smallsetminus \cN_\delta(\partial A), \\
\delta^{-1} d(x, \partial A) & \text{if } x \in A \cap \cN_\delta( \partial A), \\
0 &\text{if } x \notin A,
\end{cases}
\end{align}
and notice that $\varphi_A$ is continuous. Thus, we need to estimate the three terms in
\begin{align}\label{eq:usc potential}
\begin{split}
\left| \int g \mathbbm{1}_A \intd\mu_k - \int g \mathbbm{1}_A \intd\mu \right|
&\leqslant 
\left| \int g \mathbbm{1}_A \intd\mu_k - \int g \varphi_{A,\delta} \intd\mu_k \right| \\
& \,\, + \left| \int g \varphi_{A,\delta} \intd\mu_k - \int g \varphi_{A,\delta} \intd\mu \right| 
+ \left| \int g \varphi_{A,\delta} \intd\mu - \int g \mathbbm{1}_A \intd\mu \right|.
\end{split}
\end{align}
The third term is bounded by $\mu(\cN_\delta(\partial A)) \sup_A |g| \leqslant \ve \sup_A |g|$. For the first term, we use that since $\mu ( \partial \cN_\delta(\partial A)) = 0$, there is some $k_A$ such that for every $k \geqslant k_A$, $\mu_k (\cN_\delta(\partial A)) \leqslant \mu (\cN_\delta(\partial A)) + \ve \leqslant 2\ve$. Thus, the first term in \eqref{eq:usc potential} is bounded by $2 \ve \sup_A |g|$. Now, for the second term, we use that $g \varphi_{A,\delta}$ is continuous on $M$. Thus, from the weak-$^*$ convergence, for every $k \geqslant k_A$ (up to increasing the value of $k_A$), the second term in \eqref{eq:usc potential} is bounded by $\ve$. This prove the continuity of $\mu \mapsto \int g \mathbbm{1}_A \,\mathrm{d}\mu$, and hence the continuity of $\mu \mapsto \int g \,\mathrm{d}\mu$.
\end{proof}

From the upper semi-continuity of $\mu \mapsto P_{\mu}(T,g)$ and the compactness of $\cM(M,T)$ (see Proposition~\ref{prop: closed set}), we deduce the existence of equilibrium states. Furthermore, using Proposition~\ref{prop: averaging_measures}, we can identify the maximal value of the metric pressure.

\begin{proposition}\label{prop: var principle}
For any finite horizon Sinai billiard with negligible singularities and for any piecewise-continuous potential $g : M \to \bR$, the Variational Principle is satisfied, that is $\sup_{\mu \in \cM(M,T)} P_{\mu}(T,g) = P_{\rm top}(T,g)$.
\end{proposition}

The proof follows the same strategy as the one exposed in~\cite[Theorem~20.2.4]{katok1997introduction}. We write it for completeness. Note that in the process, we construct an equilibrium state.

\begin{proof}
Let $g$ be a piecewise continuous potential. By~\cite{Ca1}, for any invariant measure $\mu$, $P_{\mu}(T,g) \leqslant P_{\rm top}(T,g)$. We now prove the reverse inequality.

Let $0 < \chi_n < 1$ be a sequence converging to $1$. For each $n$, let $E_n$ be a finite set containing exactly one point $x_A$ of each $A \in \cP_0^{n}$ such that $e^{S_n g(x_A)} \geqslant \chi_n \sup_{A} e^{S_n g}$. Let 
\begin{align*}
\nu_n \coloneqq \left( \sum_{x \in E_n} e^{S_n g(x)} \right)^{-1} \sum_{x \in E_n} e^{S_n g(x)} \delta_x, \quad \mu_n = \frac{1}{n} \sum_{i=0}^{n-1} T^i_* \nu_n.
\end{align*}
Since $\cP_0^n$ is a partition, each $x \in E_n$ belongs to a single $A \in \cP_0^{n}$, and we get
\begin{align*}
H_{\nu_n}(\cP_0^{n}) + n \int g \intd \mu_n &= H_{\nu_n}(\cP_0^{n}) + \int S_n g \intd \nu_n = \sum_{x \in E_n} \nu_n(\{x\}) ( S_n g(x) - \log \nu_n(\{x\}) ) \\
&= \log \sum_{x \in E_n} e^{S_n g(x)} \geqslant \log \chi_n + \log \sum_{A \in \cP_n} \sup_{A} e^{S_n g} 
\end{align*}
where the last equality is due to the equality case of~\cite[Lemma~20.2.2]{katok1997introduction}. Let $q \leqslant n$ and $0 \leqslant r < q$. Let $i_r$ be the smallest integer such that $n - (r + q i_r) < q$. Since 
\begin{align*}
\cP_0^n = \cP_0^r \vee \left( \bigvee_{i=0}^{i_r-1} T^{-qi-r}\cP_0^q \right) \vee T^{r+q i_r} \cP_0^{n-(r+ qi_r)},
\end{align*}
we obtain, by sub-additivity,
\begin{align*}
H_{\nu_n}(\cP_0^n) &\leqslant H_{\nu_n}(\cP_0^r) + \sum_{i=0}^{i_r-1} H_{\nu_n}(T^{-qi-r}\cP_0^q) + H_{\nu_n}(T^{r+q i_r} \cP_0^{n-(r+ qi_r)}) \\
&\leqslant 2q \log \# \cP + \sum_{i=0}^{i_r-1} H_{\nu_n}(T^{-qi-r}\cP_0^q).
\end{align*}
Summing over $r$, we get
\begin{align*}
q H_{\nu_n}(\cP_0^n) &\leqslant 2q^2 \log \# \cP + \sum_{i=0}^{n} H_{\nu_n}(T^{-i}\cP_0^q).
\end{align*}
On the other hand, by concavity of $ \phi : t \mapsto -t \log t$, we get
\begin{align*}
H_{\mu_n}(\cP_0^q) &= \sum_{A \in \cP_0^q} \phi (\mu_n(A)) = \sum_{A \in \cP_0^q} \phi \left( \frac{1}{n} \sum_{i=0}^{n-1} T^i_*\nu_n(A) \right)\\
&\geqslant  \frac{1}{n} \sum_{i=0}^{n-1} \sum_{A \in \cP_0^q} \phi(T^i_*\nu_n(A)) = \frac{1}{n} \sum_{i=0}^{n-1} H_{\nu_n}(T^{-i} \cP_0^q). 
\end{align*}
Putting all of the above estimates together yields
\begin{align}\label{eq: variational_principle}
\frac{ \log \chi_n}{n} + \frac{1}{n}\log \sum_{A \in \cP_0^n} \sup_{A} e^{S_n g} \leqslant \frac{2q}{n} \log \# \cP + \frac{1}{q} H_{\mu_n}(\cP_0^q) + \int g \intd \mu_n.
\end{align}
By Proposition~\ref{prop: averaging_measures}, there exist $n_k \to \infty$ and $\mu_g \in \cM(M,T)$ such that $\mu _{n_k} \rightharpoonup \mu_g$. Replacing $n$ by $n_k$ in \eqref{eq: variational_principle} and taking the limit, we obtain using Lemma~\ref{prop:continuity static entropy}
\begin{align*}
P_{\rm top}(T,g) \leqslant \frac{1}{q} H_{\mu_g}(\cP_0^q) + \int g \intd \mu_g.
\end{align*}
Finally, taking the limit $q \to \infty$ yields $P_{\rm top}(T,g) \leqslant P_{\mu_g}(T,g)$.
\end{proof} 

\section{$\cM(M,T)$ is closed}\label{sect:closed}

Since $T$ is not continuous, the transfer operator $T_*$ has no reason, \emph{a priori}, to be continuous on $\cM(M)$ with respect to the weak-$^*$ topology. It is therefore unclear whether the set of fixed points $\cM(M,T)$ of $T_*$ is a closed set or not, nor if the Krylov--Bogolyubov construction of invariant measures holds. Fortunately, when the horizon is finite, the suspension flow over $T$ under the return time $\tau$ coincides with the billiard flow, a continuous flow over a compact set\footnote{In the infinite horizon case, those flows do not coincide as there are orbits making no collision.}. We can therefore relate the invariant measures with respect to $T$ to the ones of $\Phi^t$, and use the good properties of this flow. This is the strategy we follow.

\begin{proposition}\label{prop: closed set}
For any finite horizon Sinai billiard, the set of $T$-invariant probability measures $\cM(M,T)$ is a closed subset of $\cM(M)$ for the weak-$^*$ topology.
\end{proposition}

\begin{proof}
Let $\mu_n \in \cM(M,T)$ be a sequence converging to some $\mu_\infty \in \cM(M)$. Then, in the coordinates of the suspension, $\bar \mu_n \coloneqq (\int_M \tau \,\mathrm{d}\mu_n)^{-1} \mu_n \otimes \lambda$ is a sequence of $\Phi^t$-invariant probability measures, where $\lambda$ is the Lebesgue measure in the direction of the flow. Up to taking a subsequence, we can assume that $\bar \mu_n$ converges to some $\bar \mu$. Furthermore, since the billiard flow is continuous, $\bar \mu$ is a flow invariant measure. 

In particular, there exists a $T$-invariant measure $\mu$ such that $\bar \mu = (\int_M \tau \,\mathrm{d}\mu)^{-1} \mu \otimes \lambda$. We now prove that $\mu_\infty = \mu$.

Consider $\psi \in C^0(M)$ and a smooth function $\rho : \mathbbm{R} \to \mathbbm{R}_+$ supported in $[- \min \tau /2 , \min \tau /2 ]$ such that $\int_\mathbbm{R} \rho \, \mathrm{d}\lambda >0$. Consider a fundamental domain of $\Omega \cong ( M \times \mathbbm{R} )/ \sim$ containing $M \times [- \min \tau /2 , \min \tau /2 ]$ where $\sim$ is such that
\begin{align*}
(x,t) \sim (y,u) \Leftrightarrow \exists n \geqslant 0, (y,u)=(T^n x, t - \sum_{i=0}^{n-1}\tau(T^ix)) \text{ or } (x,t)=(T^n y, u - \sum_{i=0}^{n-1}\tau(T^iy)).
\end{align*}
On this domain, $\varphi(x,t) = \psi(x)\rho(t)$ is well defined and descends into a continuous function on $\Omega$, also called $\varphi$. Thus
\begin{align}\label{eq:conv}
\frac{\int_M \psi \intd\mu_n}{\int_M \tau \intd\mu_n} \int_\mathbbm{R} \rho \intd\lambda = \bar \mu_n (\varphi)  \xrightarrow[n \to \infty]{} \bar \mu(\varphi) = \frac{\int_M \psi \intd\mu}{\int_M \tau \intd\mu} \int_\mathbbm{R} \rho \intd\lambda.
\end{align}
Now, since $\mu_n(\psi)$ converges to $\mu_\infty(\psi)$, and $\int_\mathbbm{R} \rho \, \mathrm{d}\lambda >0$, we get from \eqref{eq:conv} that for every $\psi \in C^0(M)$, if $\mu_\infty(\psi) \neq 0$ then 
\begin{align*}
\frac{\mu(\tau)}{\mu_n(\tau)} \xrightarrow[n \to \infty]{} \frac{\mu(\psi)}{\mu_\infty(\psi)},
\end{align*}
and if $\mu_\infty(\psi) = 0$, then $\mu(\psi) = 0$. Therefore, there is a constant $c \geqslant \tau_{\min}/\tau_{\max} > 0$ such that for every $\psi \in C^0(M)$, $\mu(\psi) = c \, \mu_\infty(\psi)$. Since both $\mu$ and $\mu_\infty$ are probability measures, it must be that $c=1$, and so $\mu_\infty = \mu$ is invariant by $T$. Hence, $\cM(M,T)$ is a closed set.
\end{proof}

The following proposition will only be used in the proofs Proposition~\ref{prop: var principle} and Corollary~\ref{corol: existence FH}.

\begin{proposition}\label{prop: averaging_measures}
Let $\nu_n$ be a sequence in $\cM(M)$, and define $\mu_n = \frac{1}{n} \sum_{i=0}^{n-1} T_*^i \nu_n$. Then, any accumulation point of the sequence $\mu_n$ belongs to $\cM(M,T)$.
\end{proposition}

\begin{proof}
Let $\nu_n \in \cM(M)$ and define $\mu_n$ as in the proposition. Let $n_k$ be a strictly increasing sequence of integers and $\mu_\infty$ be such that $\mu _{n_k} \rightharpoonup \mu_\infty$. As in the proof of Proposition~\ref{prop: closed set}, we define for each $n$, $\bar \mu_n \coloneqq \left( \int \tau \, \intd \mu_n \right)^{-1} \mu_n \otimes \lambda$. Note that $\int \tau \, \intd \mu_n = \int S_n \tau \intd\nu_n \geqslant n \tmin$.

Fix $t \in \bR$. From the suspension coordinates, it follows that for every $\psi \in C^0(\Omega,\bR)$
\begin{align*}
|(\Phi^t)_* \bar \mu_n (\psi) - \bar \mu_n(\psi)| \leqslant \frac{2 t}{n \tmin} | \psi |_{C^0}
\end{align*}
Therefore, any accumulation point of $\bar \mu_n$ is $\Phi^t$-invariant. Hence, up to extracting a second time, we can assume that $\bar \mu_{n_k}$ converges to some $\bar \mu = (\int \tau \intd \mu)^{-1} \mu \otimes \lambda$, where $\mu \in \cM(M,T)$ since $\bar \mu$ is invariant by the flow. We now prove that $\mu_\infty = \mu$.

Let $\psi \in C^0(M,\bR)$, $\rho \in C^0(\bR,\bR)$ and $\varphi(x,t)=\psi(x)\rho(t)$ as in the proof of Proposition~\ref{prop: closed set}. Using the exact same argument, we obtain a constant $c \geqslant \tmin / \tmax > 0$ such that for every $\psi \in C^0(M,\bR)$, $\mu(\psi) = c \mu_\infty(\psi)$. Since $\mu$ and $\mu_\infty$ are probability measures, $c=1$ and hence $\mu_\infty$ is $T$-invariant.
\end{proof}

\begin{remark}
Propositions~\ref{prop: closed set} and \ref{prop: averaging_measures} actually hold in a more general context. Indeed, we only used in the proofs three ingredients: that the suspension flow over $T$ is continuous over a compact set, that $\inf \tau > 0$, and that $M$ is compact.
\end{remark}

\section{Typicality of Negligible Singularities condition}\label{sect: typicality of NS}

Assuming the genericness of the bounded complexity assumption, the genericness of the negligible singularities assumption follows from the next proposition.

\begin{proposition}
If a Sinai billiard table has bounded complexity, then its singularities are negligible.
\end{proposition}

\begin{proof}
Let $n_0 \geqslant 1$, $\varphi_0 < \pi /2$ and $\mu \in \cM(M,T)$. Therefore  
\begin{align*}
\mu(\cS_0) \leqslant \mu(\{ |\varphi| > \varphi_0 \}) = \frac{1}{n_0} \int \sum_{i=0}^{n_0-1} \mathbbm{1}_{ \{ |\varphi| > \varphi_0 \} } \circ T^i \intd\mu \leqslant \frac{1}{n_0} \sup_{M} \sum_{i=0}^{n_0-1} \mathbbm{1}_{ \{ |\varphi| > \varphi_0 \} } \circ T^i = s_0(n_0,\varphi_0).
\end{align*}
Now, Demers and Korepanov proved in~\cite[Proposition~5.6]{DK} that if the complexity is bounded, then $\inf_{(n_0,\varphi_0)} s_0(n_0,\varphi_0)=0$. Hence $\mu(\cS_0)=0$. By invariance of $\mu$, we get that $\mu(\cS)=0$.
\end{proof}

Since the problem of genericness of the bounded complexity assumption is still (officially) open, we now provide a complete proof of the genericness of the negligible singularities condition. First, identify the set of all Sinai billiard tables with $D$ scatterers and no corner nor cusp, denoted $\Bil$, with the families $(f_i)_{1 \leqslant i \leqslant D} \in C^{3+\alpha}(\bS^1,\bT^2)$, up to permutation of the indexes, such that for every $1 \leqslant i \leqslant D$ \\
\noindent
i) Each connected component of the lift of $f_i(\bS^1)$ in $\bR^2$ is the boundary of a convex set; \\
\noindent
ii) $|f'_i| >0$ is constant, $|f''_i| \neq 0$, $f_i$ turns counter-clockwise and $f_i(\bS^1)$ does not self intersect; \\
iii) For all $i \neq j$, $f_i(\bS^1) \cap f_j(\bS^1) = \emptyset$.

\medskip
We will often simply write $Q$ for an element of $\Bil$, where $Q = \bT^2 \smallsetminus \bigsqcup_{i=1}^D B_i$ and the boundary of each $B_i$ is parametrized by $f_i$. Now consider $\BilFH$ the subset of $\Bil$ made of billiard tables with finite horizon, in the sense that there are no trajectory making only grazing collisions. The sets $\Bil$ and $\BilFH$ are thus naturally endowed with the $C^{3+\alpha}$ topology and the associated metric.

We first prove that $\BilFH$ is an open subset of $\Bil$ with respect of the $C^{3+\alpha}$ topology.

\begin{proposition}\label{prop: FH is open}
$\BilFH$ is an open subset of $\Bil$ with respect of the $C^{3+\alpha}$ topology.
\end{proposition}
In order to prove this proposition, we need the following lemmas (the first of which\footnote{The statement of Lemma~\ref{lemma: control_on_flow} is strongly inspired from a presentation given by T\'oth~\cite{toth}.} will be crucial for the rest of this section).

\begin{lemma}\label{lemma: control_on_flow}
Let $Q_0 \in \Bil$, $x \in \mathring{Q}_0 \times \bS^1 \subset \bT^2 \times \bS^1 $, $\ve >0$ and $t > 0$. Then, there exist $\delta_1,\delta_2 > 0$ such that for every $Q_1 \in \Bil \cap B_{C^{3+\alpha}}(Q_0,\delta_1)$, and every $y \in  ( \mathring{Q}_1 \times \bS^1 ) \cap B(x,\delta_2)$, there exists a continuous, piecewise affine function $h_y : [0,t] \xrightarrow{\sim} [0,t]$, $\ve$-close to the identity (in the $C^0$ topology), such that
\begin{align*}
d \left( \Phi^s_0(x),\Phi^{h_y(s)}_1(y) \right) < \ve, \, \forall s \in [ 0, t ],
\end{align*}
where $\Phi_{0}$ and $\Phi_1$ are respectively the billiard flows associated to the tables $Q_0$ and $Q_1$.
\end{lemma}

\begin{proof}
Let $Q_0 \in \Bil$, $x \in \mathring{Q}_0 \times \bS^1 \subset \bT^2 \times \bS^1 $, $\ve >0$ and $t > 0$. Denote by $\Omega_0$ the phase space of $\Phi_0$. Since $\Phi_0$ is continuous on the compact set $\Omega_0 \times [0,t]$, it is uniformly continuous. In particular, there is a $\delta_2 > 0$ such that for every $y \in B_{\Omega_0}(x,\delta_2)$, $d_{\Omega_0}(\Phi_0^s(x),\Phi_0^s(y)) < \ve /2$ for every $s \in [0,t]$. Up to decreasing the value of $\delta_2$, we can assume that $B_{\Omega_0}(x,\delta_2) \Subset \mathring{Q}_0 \times \bS^1$. Thus, we now only need to control the distance between $\Phi^s_0(y)$ and $\Phi_1^s(y)$.

By the group property of the flows, we only need to prove Lemma~\ref{lemma: control_on_flow} for any $t > 0$ small enough. For the rest of the proof, we will assume that $0 < t < \min ( \tau_{0,\min} , \diam' Q_0)/3$, where $\diam' Q_0$ is the smallest diameter among the diameter of the largest disc contained in each connected component of $\bT^2 \smallsetminus Q_0$. We now provide finitely many upper bounds on $\delta_1$ such that, if they are all satisfied, Lemma~\ref{lemma: control_on_flow} holds. 

First, assume that $\delta_1  < \min ( t_{0,\min} , d(\pi(x),\partial Q_0 ), \diam' Q_0)/3$, where $\pi : \bT^2 \times \bS^1 \to \bT^2$ is the canonical projection. In that case, for any $Q_1 \in \Bil \cap B_{C^{3+\alpha}}(Q_0,\delta_1)$, we get that 
$\tau_{1,\min} \geqslant \tau_{0,\min}/3 > 0$, $B(x,\delta_2) \subset \mathring{Q}_1 \times \bS^1$ 
and both $\Phi^{[0,t]}_0(y)$ and $\Phi^{[0,t]}_1(y)$ each contains at most one collision, for any $y \in B(x,\delta_2)$. We now consider the four different cases whether collisions occur or not for any fixed $y$.

\noindent
\textit{Case 1:} Assume that $\Phi^{[0,t]}_0(y)$ and $\Phi^{[0,t]}_1(y)$ do not contain any collision. Therefore, the two trajectories coincide. The lemma holds in this case if $h_y$ is the identity.

\noindent
\textit{Case 2:} Assume that $\Phi^{[0,t]}_0(y)$ makes one collision and that $\Phi^{[0,t]}_1(y)$ does not. In order to simplify notations (also for the subsequent cases) define $\gamma_{0,i} : \bS^1 \to \bS^1$, $1 \leqslant i \leqslant D$ by 
\begin{align*}
\gamma_{0,i}(r) \coloneqq \frac{f_{0,i}'(r)}{|f_{0,i}'(r)|},
\end{align*}
where $(f_{0,i})_{1 \leqslant i \leqslant D}=Q_0$, and similarly for $Q_1$. Note that $\gamma_{0,i}$ is a $C^2$-diffeomorphism. For $\varphi_0 \in (0,\pi/2)$ to by specified latter, let
\begin{align*}
F_{0,i}(r,r_0) = \left\langle \gamma_{0,i}(r_{0}), \gamma_{0,i}(r) \right\rangle - \cos \left( \frac{\pi}{2} - \varphi_0 \right).
\end{align*}
Note that for each $r \in \bS^1$ and $1 \leqslant i \leqslant D$, there are exactly two solutions $r_{0,i}^{\pm}(r)$ to $F_{0,i}(r,r_{0,i}^{\pm}(r))=0$. We can assume that the triplets $(r_{0,i}^-(r) , r , r_{0,i}^+(r))$ keep the same orientation in $\bS^1$. Since $\varphi_0 \neq \pi/2$, $\frac{\partial F_{0,i}}{\partial r_0} (r,r_{0,i}^{\pm}) \neq 0$, by the implicit function theorem, all $r \mapsto r_{0,i}^{\pm}(r)$ are $C^2$ functions. Let
\begin{align*}
G_{0,i}^{\pm}(r) = \left\langle f_{0,i}(r_{0,i}^{\pm}(r)) - f_{0,i}(r) , \frac{f''_{0,i}(r)}{|f''_{0,i}(r)|} \right\rangle.
\end{align*}
By continuity and compactness, $\delta_{Q_0} \coloneqq \inf \{ G_{0,i}^*(r) \mid r \in \bS^1 , * \in \{-,+\} , 1 \leqslant i \leqslant D \} > 0$. Therefore, up to decreasing the value of $\delta_1$, we can assume that $\delta_1 < \delta_{Q_0}$. Thus, for every $Q_1 \in \Bil \cap B_{C^{3+\alpha}}(Q_0,\delta_1)$, the collision in $\Phi^{[0,t]}_0(y)$ is $\varphi_0$-grazing. Now, the distance between $\pi(\Phi_0^s(y))$ and $\pi(\Phi_1^s(y))$ is at most $2t \sin ( \pi/2 - \varphi_0)$. Thus, taking 
\begin{align*}
\pi/ 2 - \varphi_0 < \min \left( \frac{\pi}{2} \frac{\ve}{8t}, \frac{\ve}{8} \right),
\end{align*}
the distances in $\bT^2$ and $\bS^1$ are both at most $\ve/4$. Hence, the lemma holds if $h_y$ is the identity.

\noindent
\textit{Case 3:} Assume that $\Phi^{[0,t]}_1(y)$ makes one collision and that $\Phi^{[0,t]}_0(y)$ does not. Similarly as in Case 2, we define $\gamma_{j,i}(r,\varphi_0)$, $F_{1,i}(r,r_j,\varphi_0)$, $G^{\pm}_{1,i}(r,\varphi_0)$ and $\delta_{Q_1,\varphi_0}$, $j = 0,1$, for any $Q_1 \in \Bil \cap B_{C^{3+\alpha}}(Q_0,\delta_1)$, where we now explicit the dependence on $\varphi_0$. Our goal is to find $\delta_1$ and $\varphi_0$ such that $Q_1 \mapsto \delta_{Q_1,\varphi_0}$ is bounded from below by $\delta_{Q_0,\varphi_0}/2$ on a $\delta_1$-neighbourhood of $Q_0$. 

We start by giving an expansion of the distance between $r^{\pm}_{j,i}(r)$ and $r$ as $\varphi \to \pi /2$. First, notice that this distance can be bounded by above by the analogous quantity in the case where the curvature of the scatterers is constant (and equals the minimum). It follows that
\begin{align*}
d_{\bS^1}( r^{\pm}_{j,i}(r) , r ) \leqslant \sup_{r,i}\frac{|f'_{j,i}|}{|f''_{j,i}(r)|} \left( \frac{\pi}{2} - \varphi_0 \right).
\end{align*}
Fix some $\varphi_0$ to be chosen latter. Then, using that $F^{\pm}_{j,i}(r,r^{\pm}_{j,i}(r,\varphi_0),\varphi_0) = 0$ and expanding $\gamma_{j,i}(r^{\pm}_{j,i}(r))$ around $r$, we get
\begin{align*}
\frac{1}{2} ( r^{\pm}_{j,i}(r) - r )^2 \left( \left\langle \gamma''_{j,i}(r), \gamma_{j,i}(r) \right\rangle + \left\langle \gamma''_{j,i}( \hat r^{\pm}_{j,i}(r) ) - \gamma''_{j,i}(r), \gamma_{j,i}(r) \right\rangle \right) = \cos\left( \frac{\pi}{2} - \varphi \right) -1,
\end{align*}
where $\hat r^{\pm}_{j,i}(r) \in (r, r^{\pm}_{j,i}(r))$. Note that $\left\langle \gamma''_{j,i}(r), \gamma_{j,i}(r) \right\rangle = - \frac{|f''_{j,i}(r)|}{|f'_{j,i}|}$ and that
\begin{align*}
\left| \left\langle \gamma''_{j,i}( \hat r^{\pm}_{j,i}(r) ) - \gamma''_{j,i}(r), \gamma_{j,i}(r) \right\rangle \right|
\leqslant \frac{\Hol (f'''_{j,i})}{|f'_{j,i}|} d_{\bS^1}( r^{\pm}_{j,i}(r) , r ) \leqslant \frac{\Hol (f'''_{j,i})}{\inf_{r,i}|f''_{j,i}(r)|} \left( \frac{\pi}{2} - \varphi_0 \right).
\end{align*}
Thus,
\begin{align*}
\left| \frac{1}{2}(r^{\pm}_{j,i}(r) - r )^{2} - \frac{1}{2} \frac{|f'_{j,i}|}{|f''_{j,i}(r)|} \left( \frac{\pi}{2} - \varphi_0 \right)^{2} \right| \leqslant  C_{j,i} \, \left( \frac{\pi}{2} - \varphi_0 \right)^{2+\alpha} + C_j \left( \frac{\pi}{2} - \varphi_0 \right)^{4},
\end{align*}
where
\begin{align*}
C_{j,i} = \frac{1}{2} \frac{ \Hol (f'''_{j,i})}{\inf_{r,i}|f''_{j,i}(r)|} \quad \text{ and } \quad C_{j} =\frac{2}{4!} \sup_{r,i} \frac{|f'_{j,i}|}{|f''_{j,i}(r)|}
\end{align*}
Now, expanding $f_{j,i}(r^{\pm}_{j,i}(r))$ at $r$, we get that
\begin{align*}
&\left| G^{\pm}_{j,i}(r,r^{\pm}_{j,i}(r),\varphi_0) - \frac{1}{2} \frac{|f'_{j,i}|}{|f''_{j,i}(r)|} \left( \frac{\pi}{2} - \varphi_0 \right)^{2} \right|  \\
&\qquad \leqslant C_{j,i} \left( \frac{\pi}{2} - \varphi_0 \right)^{2+ \alpha} + \frac{1}{6} |f'''_{j,i}|_{C^0} \left( \frac{\pi}{2} - \varphi_0 \right)^{3} + C \left( \frac{\pi}{2} - \varphi_0 \right)^{4} 
\end{align*}
as well as
\begin{align*}
&\left| G^{\pm}_{1,i}(r,\varphi_0) - G^{\pm}_{0,i}(r,\varphi_0) \right| \\
&\qquad \leqslant \frac{1}{2} \left| \frac{|f'_0|}{|f''_0(r)|} - \frac{|f'_1|}{|f''_1(r)|}  \right| \left( \frac{\pi}{2} - \varphi_0 \right)^{2} + (C_{1,i}+C_{0,i}) \left( \frac{\pi}{2} - \varphi_0 \right)^{2+\alpha} + (C_1 + C_0) \left( \frac{\pi}{2} - \varphi_0 \right)^{4} \\
&\qquad \leqslant \bar \ve \left( \frac{\pi}{2} - \varphi_0 \right)^{2},
\end{align*}
where for the last inequality, we used that for every $\bar{\ve} > 0$ there exists $\bar \delta > 0$ such that for every $Q_1 \in B_{C^{3+\alpha}}(Q_0,\bar \delta)\cap \Bil$, we have
\begin{align*}
\left| \frac{|f'_0|}{|f''_0(r)|} - \frac{|f'_1|}{|f''_1(r)|}  \right| &< \frac{\bar \ve}{4} \quad ; \quad
C_{1,i} \leqslant 2 C_{0,i} \quad ; \quad
C_1 \leqslant 2 C_0 \, ,
\end{align*}
and for $\varphi_0$ close enough to $\pi /2$ so that
\begin{align*}
3C_{0,i} \left( \frac{\pi}{2} - \varphi_0 \right)^{\alpha} &< \frac{\bar \ve}{4} \quad ; \quad
3 C_0 \left( \frac{\pi}{2} - \varphi_0 \right)^{2} < \frac{\bar \ve}{4}.
\end{align*}

Now, fix $\bar \ve = \inf_{r,i}  |f'_{0,i}| / (2|f''_{0,i}(r)|)$. Therefore, for $\varphi_0$ as above and for any $r \in \bS_1$, any $1 \leqslant i \leqslant D$, and any $* \in \{ - , + \}$,
\begin{align*}
G^{*}_{1,i}(r,\varphi_0) &\geqslant G^{*}_{0,i}(r,\varphi_0) - \bar \ve \left( \frac{\pi}{2} - \varphi_0 \right)^{2} \\
&\geqslant \bar \ve \left( \frac{\pi}{2} - \varphi_0 \right)^{2} > 0.
\end{align*}
Up to decreasing the value of $\delta_1$, we can assume that $\delta_1 < \bar \delta$ and $\delta_1 < \bar \ve \left( \frac{\pi}{2} - \varphi_0 \right)^{2}$. In that case, the collision of $\Phi^{[0,t]}_1(y)$ must be $\varphi_0$-grazing. We can now conclude as in Case 2.

\noindent
\textit{Case 4:} Assume that $\Phi^{[0,t]}_0(y)$ and $\Phi^{[0,t]}_1(y)$ both make one collision. By the choice of $t$, the collisions must occur on scatterers with common indexes, say $i_0$. Denote $t_0 \in (0,t)$ (resp.$t_1$) the time at which $\Phi^{[0,t]}_0(y)$ (resp. $\Phi^{[0,t]}_1(y)$) makes a collision. Denote also $r_0$ (resp. $r_1$) such that the collision occurs at $f_{0,i_0}(r_0)$ (resp. $f_{1,i_0}(r_1)$).

Since $f''_{0,i}(r) \neq 0$ for every $r \in \bS^1$ and $1 \leqslant i \leqslant D$, for every $\bar \ve > 0$, there exists $\bar \delta > 0 $ such that $\cN_{\bar \delta}(\partial Q_0)$ does not contain any segment of length more than $\bar \ve$. Up to decreasing the value of $ \bar \delta$, we can assume that $\bar \delta \leqslant \bar \ve$.

For $\bar \ve$ to be chosen latter, assume that $\delta_1 < \bar \delta$. It follows that $|t_0 - t_1| < \bar \ve$, and 
\begin{align*}
d(f_{0,i_0}(r_0),f_{0,i_0}(r_1)) \leqslant d(f_{0,i_0}(r_0),f_{1,i_0}(r_1)) + d(f_{1,i_0}(r_1),f_{0,i_0}(r_1)) \leqslant \bar \ve + \bar \delta \leqslant 2 \bar \ve
\end{align*} 
Computing the length of the arc in $\partial Q_0$ joining $f_{0,i_0}(r_0)$ to $f_{0,i_0}(r_1)$, it comes that
\begin{align*}
d_{\bS^1}(r_0,r_1) \leqslant \frac{2 \bar \ve}{|f'_{0,i_0}|}.
\end{align*} 
Now, in order to estimate the distance between the outgoing trajectory angles, we compute
\begin{align*}
\left| \gamma_{1,i_0}(r_1) - \gamma_{0,i_0}(r_0)  \right| &\leqslant \frac{1}{|f'_{0,i_0}(r_0)|} \left( 2 \Lip(f')d_{\bS^1}(r_0,r_1) + 2d_{C^{3+\alpha}}(Q_0,Q_1) \right) \\
&\leqslant 2 \left ( \frac{|f''_{0,i_0}|_{C^0}}{|f'_{0,i_0}|_{C^0}^2} + \frac{2}{|f'_{0,i_0}|_{C^0}} \right) \bar \ve.
\end{align*}
Thus, fixing $\bar \ve$ small enough, we can assume that the angle between the two outgoing trajectories is $2 d_{\bS^1}( \gamma_{1,i_0}(r_1) , \gamma_{0,i_0}(r_0) ) < \frac{\ve}{4t}$, as well as $\bar \ve \leqslant \ve /4$. 
Finally, defining $h_y : [0,t] \to [0,t]$ by
\begin{align*}
h_y(s) = \begin{cases}
\frac{t_0}{t_1}s, & \text{if } 0 \leqslant s \leqslant t_1, \\
\frac{t-t_0}{t-t_1}s,  & \text{if } t_1 \leqslant s \leqslant t,
\end{cases}
\end{align*}
the collisions of $s \mapsto \Phi^s_0(y)$ and $s \mapsto \Phi^{h_y(s)}_1(y)$ both occur at $s=t_0$. Assuming that $\delta_1 < \bar \ve$ and $d_{C^{3+\alpha}}(Q_0,Q_1) < \delta_1$, we have for $s \leqslant t_0$,
\begin{align*}
d(\Phi^s_0(y), \Phi_1^{h_y(s)}(y)) \leqslant \bar \ve < \ve /2,
\end{align*}
and for $s \geqslant t_0$,
\begin{align*}
d(\Phi^s_0(y), \Phi_1^{h_y(s)}(y)) \leqslant \bar \ve + 2sd_{\bS^1}( \gamma_{1,i_0}(r_1) , \gamma_{0,i_0}(r_0) ) \leqslant \ve /2,
\end{align*}
finishing the proof of the lemma.
\end{proof}

For each $Q \in \Bil$ define $s_0 (Q) = \inf_{(n_0,\varphi_0)} s_0(n_0,\varphi_0)$. The next results characterise the finiteness of the horizon in terms of the value of $s_0$.
\begin{lemma}\label{lemma: characterization FH}
For every $Q \in \Bil$, $Q \in \BilFH$ if and only if $s_0(Q) < 1$. Furthermore, for every $Q \in \BilFH$, there exist $\ve > 0$ and $\delta > 0$ such that $s_0  \leqslant 1- \ve $ on $B_{C^{3+\alpha}}(Q,\delta) \cap \Bil$.
\end{lemma}

\begin{proof}
We begin with the characterisation of $\BilFH$ in terms of $s_0$. 

Let $Q \in \BilFH$. By contradiction assume that $s_0(Q) = 1$. Let $0 < \varphi_k < \pi/2$ be a sequence converging to $\pi /2$ as $k \to \infty$. Therefore, for all $n$ and $k$, there exists $x_{n,k} \in M$ maximizing $s_0(n,\varphi_k) =1$, that is, $x_{n,k}$ makes $n$ consecutive $\varphi_k$-grazing collisions. By compactness, and up to a diagonal extraction, we can assume that for every $n$, $x_{n,k}$ converges to some $x_n \in M$ as $k \to +\infty$. Note that the flow orbit of $x_n$ must be a straight line when restricted to times in $(0, n \tmin)$. In particular $x_n$ makes at least $n \frac{\tmin}{\tmax}$ grazing collisions. By compactness, we can assume that $x_n$ converges to some $x \in M$ as $n \to \infty$. Therefore, for every $t > 0$, $\Phi^{[0,t]}(x)$ is arbitrarily close to a straight line. Hence, it is a straight line. In particular, $x$ makes only grazing collisions, contradicting the finite horizon assumption on $Q$.

Conversely, if $Q \in \Bil \smallsetminus \BilFH$, then there is a point $x \in M$ making only grazing collisions. In particular $s_0(n,\varphi)=1$ for all $n$ and $\varphi$. Hence $s_0(Q) =1$. 

\medskip
Let $Q_0 \in \BilFH$. By contradiction, assume that for all $\ve > 0$ and $k \geqslant 1$, there exists $Q_k^{\ve} \in B_{C^{3+\alpha}}(Q_0, 1/k) \cap \Bil$ such that $s_0(Q_k^{\ve}) > 1 - \ve$. Remark that for all $n_0, \varphi_0$, $s_0(Q_k^{\ve},n_0,\varphi_0) > 1 - \ve$, therefore there exists a point $x \in M_{k,\ve}$ maximizing $s_0(Q_k^{\ve},n_0,\varphi_0)$, thus making at least $(1-\ve)n_0$ grazing collisions among $n_0$ consecutive collisions. In particular, it follows that there is a point $x'$ (belonging to $\{ x ,T_k x , \ldots T_k^{n_0}x \}$) making at least $\frac{n_0 s_0(Q_k,n_0,\varphi_0)}{n_0 ( 1 - s_0(Q_k,n_0,\varphi_0)) +1} \geqslant \frac{1 - \ve}{2 \ve}$ consecutive grazing collisions, as long as $1/n_0 < \ve$.

Fix $t>0$ and chose some $0 < \varphi < \pi/2$ such that
\[ \left| \frac{\pi}{2} - \varphi \right| < \ve \left( \frac{n(n+1)}{2} \frac{\pi}{2} t \right)^{-1}.\]
Therefore, any flow orbit of length at most $t$ making at most $n$ collisions, all of which being $\varphi$-grazing, deviates from a straight line by at most $\ve$, when seen in $\bT^2$.

For each $k$ and $\ve$, let $x_{k,\ve} \in M_{k,\ve}$ be a maximizer of $s_0(Q_k^{\ve},n,\varphi)$. As shown above, there is a point $x'_{k,\ve}$ making at least $\frac{1-\ve}{2 \ve}$ successive $\varphi$-grazing collisions. These points may get arbitrary close to $\partial Q_0$, preventing from using Lemma~\ref{lemma: control_on_flow}. For $k$ large enough so that $1/k < \tmin /3$, there exists $x''_{k,\ve} = \Phi^{-t_{k,\ve}}_{k,\ve}(x'_{k,\ve})$ such that there are no collision between $x''_{k,\ve}$ and $x'_{k,\ve}$, and $d(x''_{k,\ve},\partial Q_0) \geqslant \tmin /3$.
Note that $t_{k,\ve} > 0$ can be chosen to be smaller than 
\[ t' \coloneqq \sqrt{ \frac{\tmin}{3}\left( \frac{\tmin}{3} + \frac{2}{\inf_{r,i} |f_{0,i}''(r)|} \right)},\]
whenever $1/k < \inf_{r,i} |f_{0,i}''(r)| / 2$. By compactness, we can assume that $x''_{k,\ve}$ converges to some $x''_{\ve}$ as $k \to \infty$. Since, $d(x''_{\ve},\partial Q_0) > 0$, we can apply Lemma~\ref{lemma: control_on_flow} to $Q_0$, $x''_{\ve}$, $\ve$ and $t$, and obtain some positive $\delta_1$ and $\delta_2$. Denote $t_\ve = \min \left( t,  \frac{\tmin}{3} \frac{1-\ve}{2\ve} \right)$. Therefore, when $k$ is large enough so that $1/k < \delta_1$ and $d(x''_{k,\ve},x''_{\ve})< \delta_2$, $\pi (\Phi_0^{[0,t_\ve]}(x''_{\ve}))$ deviates from a straight line by at most $2\ve$.

By compactness, let $\ve_n \to 0$ be such that $x''_{\ve_n}$ converges to some $x''$. By uniform continuity of $\Phi_0$ on $[0,t] \times \Omega_0$, $\pi(\Phi_0^{[0,t]})$ coincides with a straight line, that is, $x''$ makes only grazing collisions. In particular, the orbit of $x''$ makes at least $t/\tmax$ consecutive grazing collisions. Since this construction holds for every $t>0$, we get that $s_0(Q_0)=1$, a contradiction.
\end{proof}

\begin{proof}[Proof of Proposition~\ref{prop: FH is open}]
Let $Q_0 \in \BilFH$. Denote by $\tilde \kappa_{\min}$ the minimal value of $|f''_i|$, $1 \leqslant i \leqslant D$, and by $V > 0$ the minimal value of $|f'_i|$, $1 \leqslant i \leqslant D$. By compactness, $\tilde \kappa_{\min} > 0$. Furthermore, by definition of $\BilFH$, $\tmin > 0$. Let $\delta_1 = \min \{ \tilde \kappa_{\min}, V, \tmin \} / 3$. In particular, for every $Q \in B_{C^{3+\alpha}}(Q_0,\delta_1) \cap \Bil$, the scatterers of $Q$ are pairwise disjoint, do not self intersect and are strictly convex sets (not reduced to a point). Now, let $\ve > 0$ and $\delta > 0$ as provided by Lemma~\ref{lemma: characterization FH}. Thus, up to decreasing the value of $\delta$ so that $\delta < \delta_1$, 
\[ B_{C^3}(Q_0, \delta) \cap \Bil \subset \BilFH, \]
finishing the proof.
\end{proof}

Now, given a table $Q \in \BilFH$, we characterize the measures $\mu \in \cM(M,T)$ such that $\mu(\cS) \neq 0$. Since $\cS$ is invariant, the conditional measures $\mu_{\cS}$ and $\mu_{\cS^c}$ are also invariant, and $\mu_{\cS}(\cS)=1$. Thus, we only need to characterize measures $\mu$ such that $\mu(\cS)=1$. Denote $\cM_{\cS}(M,T)$ the set of those measures.

\begin{lemma}\label{lemma: structure of measures on singularities}
If $\mu \in \cM_{\cS}(M,T)$, then $\mu$ is supported by an at most countable family of periodic orbits, each one making at least one grazing collision.
\end{lemma}

\begin{proof}
Let $\mu \in \cM_{\cS}(M,T)$. Then, by invariance, $\mu(\cS) = 1$ implies $\mu(\cS_0) > 0$. Let $(x_k)_{1 \leqslant k \leqslant k_n}$ be the points in $\cS_n$ where multiple curves intersect. Let 
\begin{align*}
X_n \coloneqq \{ x \in \cS_0 \mid \exists \, 1 \leqslant k \leqslant k_n, \, \exists \, 0 \leqslant j \leqslant n, \, x_k= T^{-j}x \}. 
\end{align*}
Therefore $X_n$ is finite and $W_n \coloneqq \cS_0 \smallsetminus X_n$ is such that the sets $(T^{-i}W_n)_{0 \leqslant i \leqslant n}$ are pairwise disjoint. Let $W \coloneqq \bigcap_{n \geqslant 1} W_n$. Therefore, all the $T^{-n}W$, $n\geqslant 0$ are pairwise disjoint. In particular $\mu(W) = 0$. Finally, note that $\cS_0 \smallsetminus W$ is at most countable.

Therefore, $\mu_{\cS_0}$ is an atomic measure on $\cS_0$. Since $\mu$ is a probability measure, by invariance these atoms must be supported by periodic points. 
\end{proof} 

From Lemma~\ref{lemma: structure of measures on singularities}, it follows immediately that a billiard table has negligible singularities if and only if there are no periodic orbit making grazing collisions. We thus introduce the sets, for any $k \geqslant 3$,
\begin{align*}
\cU_k \coloneqq \{ Q \in \BilFH \mid \bigcup_{i=3}^k \Fix \, T^i \cap \cS_0 = \emptyset \}.
\end{align*}
Therefore, $\bigcap_{k \geqslant 3} \cU_k$ is the set of billiard tables with negligible singularities. The typicality of this condition then follows from the next proposition.

\begin{proposition}\label{prop: U_k are open and dense}
For every $k \geqslant 3$, $\cU_k$ is open and dense in $\BilFH$ with respect to the $C^{3+\alpha}$ topology. In particular, the set of finite horizon Sinai billiard tables with negligible singularities is a dense $G_\delta$ set.
\end{proposition}
Note that billiard tables admitting grazing periodic orbits are suspected to be dense. A partial result in that direction can be found in~\cite{osterman2021}.

\begin{proof}
Let $k \geqslant 3$. We first prove that $\cU_k$ is open. Let $Q_0 \in \cU_k$. By contradiction, assume that for every $n \geqslant 1$, there exists $Q_n \in B_{C^{3+\alpha}}(Q_0,1/n)$ such that $Q_n \notin \cU_k$. Denote $\Phi_{n}$ the billiard flow associated with $Q_n$, and $T_n : M_n \to M_n$ its collision map. Therefore, there is a periodic point $x_n \in M_n \subset \bT^2 \times \bS^1$ of period at most $k$ with respect to $T_n$ making at least one grazing collision. For $n$ large enough so that $d_{C^{3+\alpha}}(Q_0,Q_n) < \tmin /3$, denote by $x'_n$ a point along the flow orbit of $x_n$ such that $d(\pi(x'_n),\partial Q_0) \geqslant \tmin /3$. Let $0 <p_n < k \tmax$ be the smallest period of $x'_n$ under $\Phi_{n}$, and $t_n$ be such that $\Phi_{n}^{t_n}(x'_n)$ is a grazing collision. By compactness, up to extracting, we can assume that $x'_n$ converges to some $x' \in \bT^2 \times \bS^1$ such that $d(\pi(x'),\partial Q_0) > 0$, and, up to extracting further, that $p_n$ and $t_n$ converge to $p, t' \in [0, k \tmax]$ respectively. Now, applying Lemma~\ref{lemma: control_on_flow} for $Q_0$, $x'$, $\ve$ and $t=k \tmax$ we obtain some $\delta_1, \delta_2 >0$. For any $n$ large enough, $d_{C^{3+\alpha}}(Q_n,Q_0)< \delta_1$ and $d(x'_n,x')< \delta_2$. Therefore
\begin{align*}
d(x',\Phi_{0}^{p_n}(x')) \leqslant d(x',x'_n) + d(\Phi_{n}^{p_n}(x'_n),\Phi_{0}^{p_n}(x')) \leqslant 2 \ve
\end{align*}
Thus $x'$ is a periodic orbit of period $p$ for $\Phi_{0}$. Denote $x = \Phi_0^{t'}(x')$. Now, since
\begin{align*}
d(\Phi_{0}^{t_n}(x'), M_0) &\leqslant d(\Phi_{0}^{t_n}(x'), \Phi_{n}^{t_n}(x'_n)) + d(\Phi_{n}^{t_n}(x'_n),M_n) + d(M_n,M_0) \\
&\leqslant \ve + 1/n \leqslant 2\ve,
\end{align*} 
and by continuity of $\Phi_{0}$, $\Phi_{0}^{t'}(x')$ is a collision point. This is a grazing collision since the flow trajectory of $x'$ for times near $t'$ is arbitrarily close to segments of length at least $\tmin / 3$. This contradicts $Q_0 \in \cU_k$, and thus there exists $n$ such that $B_{C^{3+\alpha}}(Q_0,1/n) \subset \cU_k$.

\medskip
We now prove the density of $\cU_k$ in $\BilFH$. Let $Q_0 \in \BilFH$. Let $\delta$ be so small that $Q \mapsto \# \cP_{Q}$ is uniformly bounded from above by some finite $C \geqslant 2$. Denote $\delta'=\delta / 2 C^{k+1}$. Given a sequence $Q_n \in \BilFH$, denote  
\begin{align*}
P_n &= \bigcup_{i=3}^k \Fix \, T_n^i \subset M_n \subset \bT^2 \times \bS^1, \\
R_n &= \{ x \in M_n \mid x \in P_n \text{ such that } \cO_n(x) \cap \partial M_n = \emptyset \},
\end{align*}
where $\cO_n(x)$ denotes the orbit of $x$ under $T_n$, the collision map associated with $Q_n$. We now proceed algorithmically by constructing a finite sequence of tables $Q_n$ and sets $R'_n \subset R_n$ with $d_{C^{3+\alpha}}(Q_n,Q_{n+1}) < \delta'$ and $\# R'_n \geqslant n$ until $R'_n = P_n$. These constructions have to stop at some point since $d_{C^{3+\alpha}}(Q_n,Q_0) < \delta$ for every $n \leqslant 2C^{k+1}$, and thus $n \leqslant \# R'_n \leqslant \# P_n \leqslant C^{k+1}$.

Initiate $R'_0$ by setting $R'_0=R_0$. Assume $Q_n$ and $R'_n$ have been constructed. If $R'_n = P_n$, then $Q_n \in \cU_k$ and we stop the construction here. Otherwise, there is a point $x \in P_n \smallsetminus R'_n$. Since $R'_n \subset R_n$, the point $x_n$ is periodic of period at most $k$ with respect to $T_n$ and makes at least one grazing collision. Without loss of generality, we assume that $x$ is a regular collision. Denote $\pi : \bT^2 \times \bS^1 \to \bT^2$ the canonical projection. Write $\cO_n(x_n) \supset \{ y_1, \ldots , y_{i_1}, z_1 , \ldots , z_{i_2} \}$ the points of the orbit where the collisions are grazing, where for every $i$, $\pi(y_i) \in \pi(R'_n)$ and $\pi(z_i) \notin \pi(R'_n)$. Let $U_1$ be a neighbourhood of the $\pi(y_i)$'s in $\partial Q_n \subset \bT^2$ such that $\bar U_1 \cap \left( \pi(\{ z_1, \ldots z_{i_2} \} \cup R'_n) \smallsetminus \pi( \{ y_1 , \ldots , y_{i_2} \}) \right) = \emptyset $. Let $U_2$ be a neighbourhood of the $\pi(z_i)$'s in $\partial Q_n$ such that $\bar U_2 \cap \pi(R'_n) = \bar U_1 \cap \bar U_2 = \emptyset$. We construct $Q_{n+1}$ by perturbing $\partial Q_n$ only on $U_1 \cup U_2$. On $U_2$ we shrink the scatterers so that the $\pi(z_i)$'s are not part of $\partial Q_{n+1}$ and such that $d_{C^{3+\alpha}}(Q_{n+1},Q_n)<\delta'$. Note that the flow trajectory of $x$ is unchanged. The perturbation on $U_1$ needs more care, we proceed as follows.

Let $y$ be a point of $R'_n \smallsetminus \{ y_1 , \ldots , y_{i_1} \}$ such that $\pi(y) \in \pi(R'_n)$. Since the orbit of $y'$ makes only regular collisions, by the hyperbolic fixed point theorem, there is a $\delta_{y} > 0$ such that for any $Q \in B_{C^{3+\alpha}}(Q_n,\delta_{y})$, there corresponds a unique periodic point $y'$ whose orbit  contains only regular collisions, with the same period as $y$ (with respect to the collision maps), and whose orbits with respect to the flows are close together. There are finitely many such points $y$. Let $\delta'$ be the smallest $\delta_{y}$. Up to decreasing the value of $\delta_1$, we can assume $\delta_1 < \delta'$. Finally, pick $Q_{n+1} \in B_{C^{3+\alpha}}(Q_n,\delta_1)$ such that $\partial Q_{n+1}$ coincides with $\partial Q_n$ outside of $U_1 \cup U_2$ and such that the orbit of $x$ in $Q_{n+1}$ does not make grazing collision. Set $R'_{n+1}$ to be the union of the orbit of $x$ and $R'_n$, where the orbits of each $y$ are replaced by the orbits of $y'$. In particular $\# R'_{n+1} = \# R'_n + \# \cO_{n+1}(x) \geqslant \# R'_n + 1$.
\end{proof}

\section{Connection between Negligible Singularities and $s_0$}\label{sect: NS and s0}

In this section we prove that given a finite horizon billiard table $Q$, the singularities are negligible if and only if $s_0(Q) = 0$. In particular, the negligible singularities condition not only implies the sparse recurrence (with respect to $g \equiv 0$) as introduced by Baladi and Demers~\cite{BD2020MME}, but also implies its genericness in the case of tables with $C^{3 + \alpha}$ boundary. To do so, we rely on the notion of orbit capacity (as introduced by Shub and Weiss in~\cite{shub1991} for continuous transformations).

\begin{definition}\label{def: ocap}
Given a set $E \subset M$, the \emph{orbit capacity} of $E$ (with respect to $T$) is defined by
\begin{align*}
\ocap(E) \coloneqq \lim_{n \to \infty} \frac{1}{n} \sup_{M} S_n \mathbbm{1}_{E} = \inf_{n \geqslant 1} \frac{1}{n} \sup_{M} S_n \mathbbm{1}_{E}.
\end{align*}
\end{definition}
Notice that $\ocap(E)$ is well defined because $n \mapsto \sup_{M} S_n \mathbbm{1}_{E}$ is subadditive.

This definition allows us to state the condition of negligible singularities in another way.
\begin{proposition}
If $E \subset M$ is a closed set, then $\ocap(E) = \sup \{ \mu(E) \mid \mu \in \cM(M,T) \}$, and the sup is achieved. In particular, the singularities are negligible if and only if $\ocap(\cS_0)=0$. 
\end{proposition}
\begin{proof}
Let $\mu \in \cM(M,T)$. Then, for each $n \geqslant 1$,
\begin{align*}
\mu(E) = \int_{M} \mathbbm{1}_E \intd \mu = \frac{1}{n} \int_M S_n \mathbbm{1}_E \intd \mu \leqslant \frac{1}{n} \sup_{M} S_n \mathbbm{1}_{E}.
\end{align*}
Taking the limit $n \to \infty$, we get that $\mu(E) \leqslant \ocap(E)$. We now prove the reverse inequality.

Let $n \geqslant 1$. Since $\frac{1}{n} S_n \mathbbm{1}_E$ takes values in $\{ 0, 1/n, \ldots 1\}$, there exists $x_n$ such that $\frac{1}{n} S_n \mathbbm{1}_E (x_n) = \sup_{M} \frac{1}{n} S_n \mathbbm{1}_E$. Letting $\nu_n = \frac{1}{n}\sum_{i=0}^{n-1} T^i_* \delta_{x_n}$, we get that $\nu_n (E) = \sup_M \frac{1}{n} S_n \mathbbm{1}_E$. In particular $\nu_n(E) \to \ocap(E)$ as $n$ goes to infinity. By Proposition~\ref{prop: averaging_measures}, the sequence $\nu_n$ admits a converging subsequence $\nu_{n_k}$, which converges to some $\nu \in \cM(M,T)$. Now, since $E$ is closed, by the Portemanteau theorem, we get that $\limsup \nu_{n_k}(E) \leqslant \nu(E)$. Hence $\ocap(E) \leqslant \nu(E)$.

The last part of the statement follows since $\cS_0$ is closed.
\end{proof}

We can also use Definition~\ref{def: ocap} to relate $s_0(Q)$ to the orbit capacity of neighbourhoods of $\cS_0$.
\begin{proposition}
For any finite horizon Sinai billiard table $Q$, the following limit holds
\begin{align*}
s_0(Q) = \lim_{\delta \to 0} \ocap( \cN_{\delta}(\cS_0)),
\end{align*}
where $\cN_\delta(\cdot)$ denotes the $\delta$-neighbourhood of a set.  
\end{proposition}

\begin{proof}
Recall that $s_0(Q) = \inf_{(n,\varphi)} s_0(n,\varphi)$. Clearly, $s_0(Q) \leqslant \inf_{\varphi \in (0,\pi/2)} \inf_{n \geqslant 1} s_0(n,\varphi)$. We prove the converse inequality. Let $\ve >0$ and let $n_0,\varphi_0$ be such that $s_0(n_0,\varphi_0) \leqslant \inf_{(n,\varphi)} s_0(n,\varphi) + \ve$. In particular, $\inf_{n} s_0(n,\varphi_0) \leqslant \inf_{(n,\varphi)} s_0(n,\varphi) + \ve$, and thus we get that $\inf_{\varphi} \inf_{n} s_0(n,\varphi) \leqslant \inf_{(n,\varphi)} s_0(n,\varphi) + \ve$, for any positive $\ve$.

Now, notice that $\inf_{n \geqslant 1} s_0(n,\varphi) = \ocap( \cN_{\delta}(\cS_0) )$, where $\delta = \pi/2 - \varphi$. Finally, since the map $\delta \mapsto \ocap( \cN_{\delta}(\cS_0))$ is non-decreasing, we get that
\begin{align*}
\lim_{\delta \to 0} \ocap( \cN_{\delta}(\cS_0) ) = \inf_{\delta \in (0,\pi/2)} \ocap( \cN_{\delta}(\cS_0)) = \inf_{\varphi \in (0, \pi/2)} \inf_{n \geqslant 1} s_0(n,\varphi) = s_0(Q),
\end{align*}
as claimed.
\end{proof}

The final bridge between negligible singularities and the smallness of $s_0(Q)$ comes from the next result.
\begin{proposition}\label{prop: NS and zero s0}
For any finite horizon Sinai billiard, the following limit holds
\begin{align*}
\lim_{\delta \to 0} \ocap( \cN_\delta(\cS_0) ) = \ocap(\cS_0).
\end{align*}
Therefore $\ocap(\cS_0) = \inf_{(n,\varphi)}s_0(n,\varphi)$. In particular, the singularities are negligible if and only if $\inf_{(n,\varphi)} s_0(n,\varphi) = 0$.
\end{proposition}

Before proving Proposition~\ref{prop: NS and zero s0}, we need the following technical lemma.
\begin{lemma}\label{prop: small ocap}
For any finite horizon Sinai billiard, there exists a (grazing) periodic point $x \in M$ of period say $p$, such that $\ocap(\cS_0) = \nu_x(\cS_0)$, where $\nu_x= \frac{1}{p}\sum_{i=0}^{p-1} T^i_{*} \delta_x \in \cM(M,T)$. In particular, $\ocap(\cS_0) < 1$ and $\ocap(\cS_0) \in \mathbbm{Q}$.
\end{lemma}

\begin{proof}[Proof of Lemma~\ref{prop: small ocap}]
Since the case $\ocap(\cS_0)=0$ is immediate, we consider the case $\ocap(\cS_0) > 0$. Since $\cS_0$ is closed, there is a measure $\mu \in \cM(M,T)$ such that $\mu(\cS_0) = \ocap(\cS_0)$. Consider the decomposition into a sum of conditional probabilities $\mu = \mu(\cS) \mu _{\cS} + (1 - \mu(\cS)) \mu _{\cS^c}$. Since $\mu(\cS_0) = \mu(\cS) \mu _{\cS}(\cS_0) \leqslant \mu _{\cS}(\cS_0)$, we get that, by maximality of $\mu$, it must be that $\mu(\cS)=1$. That is $\mu = \mu _{\cS}$. 

By Lemma~\ref{lemma: structure of measures on singularities}, the measure $\mu$ is an atomic measure, where atoms belong to grazing periodic orbits. In other words, there are periodic points $x_i \in \cS_0$ of period $p_i$ and a sequence $\omega_i \in (0,1)$ such that $\sum_{i \geqslant 0} \omega_i = 1$, and $\mu = \sum_{i \geqslant 0} \omega_i \, \delta_{\cO(x_i)}$, where $\delta_{\cO(x_i)} \coloneqq \frac{1}{p_i}\sum_{j=0}^{p_i-1} T^j_* \delta_{x_i} \in \cM(M,T)$. If for every $i \geqslant 0$, $\delta_{\cO(x_i)}(\cS_0) < \ocap(\cS_0)$, then we would get that $\mu(\cS_0) < \ocap(\cS_0)$. Thus, there exists some $i_0$ such that $\delta_{\cO(x_i)}(\cS_0) \geqslant \ocap(\cS_0)$. The $T$-invariance of $\delta_{\cO(x_i)}$ yields the desired equality with $x = x_i$ and $\nu_x = \delta_{\cO(x_i)}$.
\end{proof}

\begin{proof}[Proof of Proposition~\ref{prop: NS and zero s0}]
Let $\delta > 0$. Thus $\cN_\delta(\cS_0) \supset \cS_0$, and it follows from the definition of orbit capacities that $\ocap(\cN_\delta(\cS_0)) \geqslant \ocap(\cS_0)$. Since the limit $\delta \to 0$ exists, we proved one inequality.

Let $\ve > 0$. We construct some $\delta > 0$ such that $\ocap(\cN_\delta(\cS_0)) \leqslant \ocap(\cS_0) + \ve$. Thanks to Lemma~\ref{prop: small ocap}, we can assume $\ve$ to be small enough so that $\ocap(\cS_0) + \ve < 1$. Let $N$ be large enough so that $\frac{1}{N}\sup_{M} S_N \mathbbm{1}_{\cS_0} \leqslant \ocap(\cS_0) + \ve$. Let $x \in M$. We distinguish between the two cases whether $x \in \cS_N$ or not. 

If $x \in M \smallsetminus \cS_N$, then $\frac{1}{N}\sup_{M} S_N \mathbbm{1}_{\cS_0}(x) = 0$. Now, the definition of $\cS_N$ yields \[ d_x \coloneqq d(\{T^i x\}_{0 \leqslant i < N}, \cS_0)/2 > 0.\] 
Thus, $\frac{1}{N}S_N \mathbbm{1}_{\cN_{d_x}(\cS_0)}(x) = 0$. Furthermore, by the continuity of $T,T^2,\ldots, T^{N-1}$ on $\cM_0^N(x)$ (the atom of $\cM_0^N$ containing $x$), there exists a $d'_x > 0$ such that for all $0 \leqslant i < N$ and $y \in B(x,d'_x)$, $T^i(y) \notin \cN_{d_x}(\cS_0)$. In particular we get that
\begin{align*}
\frac{1}{N} S_N \mathbbm{1}_{\cN_{d_x}(\cS_0)}(y) = 0, \quad \forall y \in B(x,d'_x).
\end{align*}

We now consider the case $x \in M \smallsetminus \cS_N$. Let $t_x > 0$ be such that $T^N(x) = \Phi^{t_x}(x)$. Let $\tilde d_x >0$ be small enough so that for any $A \in \cM_0^N$, $A \cap B(x,\tilde d_x) \neq \emptyset $ if and only if $x \in \bar A$. Note that there is at most $K N$ such atoms. 
Consider $A \in \cM_0^N$ such that $x \in \bar A$. Define $d_{x,A} \coloneqq d(\{ T^i x \mid 0\leqslant i < N, \, T^i x \notin \cS_0 \} , \cS_0) /2 > 0$ (the set of $i$'s is non-empty because $\ocap(\cS_0)+\ve < 1$ and the choice of $N$). Now, by continuity of the billiard flow, we can find some $d'_{x,A} > 0$ (which we assume smaller than $\tilde d_x$), such that for every $y \in \bar A \cap B(x,d'_{x,A})$ and for any $t \in [0,t_x]$ such that $\Phi^t(x)$ is a regular collision, then there is $t'$ near $t$ such that $\Phi^{t'}(y) \in M$ is a regular collision, at distance at least $d_{x,A}/2$ from $\cS_0$. Notice that there are no restriction on the behaviour of the flow orbit of $y$ near a grazing collision on the orbit of $x$: the corresponding collision of $y$, if it exists, can be grazing or not, or the scatterer can even be missed. In any case, we get that 
\begin{align*}
\frac{1}{N} S_N \mathbbm{1}_{\cN_{d_{x,A}}(\cS_0)}(y) \leqslant \frac{1}{N} S_N \mathbbm{1}_{\cS_0}(x) \leqslant \ocap(\cS_0) + \ve, \quad \forall y \in B(x, d'_{x,A}).
\end{align*}
Defining $d_x = \min \{ d_{x,A} \mid A \in \cM_0^N, \, x \in \bar{A} \} >0$ and similarly $d'_x = \min_{A} d'_{x,A}$ yield
\begin{align*}
\frac{1}{N} S_N \mathbbm{1}_{\cN_{d_{x}}(\cS_0)}(y) \leqslant \ocap(\cS_0) + \ve, \quad \forall y \in B(x, d'_{x}).
\end{align*} 
Finally, by compactness of $M$, there is a finite collection of points $x_1,\ldots,x_\ell$ such that $M = \bigcup_{i=1}^\ell B(x_i,d'_{x_i})$. Thus, by setting $\delta = \min \{ d_{x_i} \mid 1 \leqslant i \leqslant \ell \}$ we get
\begin{align*}
\ocap(\cN_{\delta}(\cS_0)) \leqslant \frac{1}{N} \sup_{M} S_N \mathbbm{1}_{\cN_{\delta}(\cS_0)} \leqslant \ocap(\cS_0) + \ve.
\end{align*}
Now, since $\delta \mapsto \ocap(\cN_\delta(\cS_0))$ is non-decreasing, the claim follows.
\end{proof}

\section{Regularity of the metric pressure and consequences}\label{sect:regularity_and_consequences}

In this section we consider general finite horizon Sinai billiards (which may not have negligible singularities). We prove the bound on the defect of upper semi-continuity of the pressure map in Corollary~\ref{prop:defect of USC} (generalizing Corollary~\ref{prop:USC of KS entropy}) as well as some consequences, such that existence of equilibrium states under a mild assumption the potential (Corollary~\ref{corol: existence FH}) or as soon as there are maximizing sequences (Corollary~\ref{corol: max sequence}), upper semi-continuity of the entropy map (Theorem~\ref{thm: USC KS entropy}) and equidistribution of the periodic orbits (Corollary~\ref{corol: equidist of per orbits}).
 
\subsection{Bound on the defect of upper semi-continuity of the metric pressure}\label{subsect:defect_USC}

Before investigating the smoothness of the pressure map, we give a preliminary result estimating the size of eventual discontinuities of each of the \emph{static pressure maps} $\mu \mapsto P_{\mu}(g,\cP_0^n) \coloneqq \frac{1}{n} H_\mu(\cP_0^n) + \int g \intd \mu$. Note that by invariance 
\begin{align}\label{eq: def_conditional_pressure}
P_{\mu}(g,\cP_0^n) \coloneqq \frac{1}{n}  \left( H_\mu(\cP_0^n) + \int S_n g \intd \mu \right) .
\end{align}
In the case where $\mu \in \cM(T)$ is not invariant by $T$, we define the static pressure of $\mu$ by \eqref{eq: def_conditional_pressure}.

\begin{proposition}\label{prop: diff of static entropies}
For every $\mu \in \cM(M,T)$, every $\cM_0^1$-continuous potential $g : M \to \bR$, every small enough $\ve >0$ and every large enough $n$, there is a neighbourhood $V_{\ve,n} \subset \cM(M)$ of $\mu$ such that for every $\nu \in V_{\ve,n}$,
\begin{align*}
P_{\nu}(g,\cP_0^n) - P_{\mu}(g,\cP_0^n)  &\leqslant \frac{\mu(\cS)}{n} \log \sum_{A \in \cP_0^n} \sup_{A} e^{S_n g} - \int_{\cS} g \intd\mu + R_{\ve,n}(g),
\end{align*}
where
\begin{align*}
R_{\ve,n}(g) = \ve \left( 4 \left| \frac{1}{n} \log \sum_{A \in \cP_0^n} \sup_{A} e^{S_n g} \right| + \sup |g| + 4 \frac{\log \# \cP_0^n}{n} \right) + \frac{2 \log 2 + e^{-1} + 5\ve}{n} + 2 \ve.
\end{align*}
\end{proposition}

To prove the above proposition, we split the contributions to $P_\mu(g,\cP_0^n)$ of a neighbourhood of $\partial \cP_0^n$ and its complement. To do so, we rely on the use of conditional entropy and conditional pressure. Recall that for any two finite partitions $\cA$ and $\cC$, the conditional entropy of $\cA$ with respect to $\cC$ and $\mu$ is defined by
\begin{align*}
H_\mu(\cA \mid \cC) \coloneqq \sum_{C \in \cC} \mu(C) H_\mu(\cA \mid C) \quad \text{with } H_\mu(\cA \mid C) \coloneqq H_{\mu_C}(\cA),
\end{align*} 
where $\mu_C$ is the conditional probability of $\mu$ on $C$. By convention, if $\mu(C)=0$, we set $H_\mu(\cA \mid C) = 0$. Analogously, we also denote 
\begin{align*}
P_{\mu}(g,\cA \mid \cC) \coloneqq \sum_{C \in \cC} \mu(C) P_{\mu}(g,\cA \mid C) \quad \text{with } P_{\mu}(g,\cA \mid C) \coloneqq  \frac{1}{n} \left( H_{\mu_C}(\cA) + \int S_n g \intd \mu_C \right) .
\end{align*}

In order to estimate the contribution away from singularities in the static entropy part, we need the following technical lemma.
\begin{lemma}\label{lemma: diff of x log x}
For all $x, y \in [0,1]$ if $|x-y| \leqslant 1/2$, then $|x\log x - y\log y| \leqslant -|x-y| \log |x-y|$.
\end{lemma}

\begin{proof}[Proof of Lemma~\ref{lemma: diff of x log x}]
Let $f_\delta(x)= -x\log x + (x+\delta)\log (x+\delta )$, for $0 \leqslant \delta \leqslant 1/2$ and $x \in [0, 1- \delta ]$. Differentiating, we get $f'_\delta(x) = \log ((x+\delta)/x) > 0$ on $(0,1- \delta )$. Thus
\[ \delta \log \delta = f_\delta(0) \leqslant f_\delta \leqslant f_\delta(1-\delta ) = -(1-\delta)\log (1- \delta). \]
Let $g(\delta) = -\delta \log \delta + (1-\delta )\log (1-\delta )$, for $0 \leqslant \delta \leqslant 1/2$. Differentiating twice, we get $g''(\delta) = \frac{2\delta -1}{\delta (1 - \delta )} < 0$ on $(0,1/2)$. Thus $g$ is concave on $[0,1/2]$ and $g(\delta) \geqslant 2\delta g(1/2) + (1- 2\delta )g(0) = 0$. In other words, $-(1-\delta ) \log (1-\delta) \leqslant - \delta \log \delta$. Taking $x \leqslant y$ and $\delta=y-x$ ends the proof.
\end{proof}

\begin{proof}[Proof of Proposition~\ref{prop: diff of static entropies}]
Let $\ve > 0$ and $g : M \to \bR$ a $\cM_0^1$-continuous potential. First, remark that for any partition $\alpha$ of $M$ into two elements, we have on the one hand, for all $\nu \in \cM(M)$ and $n \geqslant 1$,
\begin{align*}
H_\nu(\cP_0^n \vee \alpha) = H_\nu(\alpha) + H_\nu(\cP_0^n \mid \alpha),
\end{align*} 
where $0 \leqslant H_\nu(\alpha) \leqslant \log 2$ by~\cite[Corollary~4.2.1]{Walters82ergth}, and on the other hand,
\begin{align*}
H_\nu(\cP_0^n \vee \alpha) = H_\nu(\cP_0^n) + H_\nu(\alpha \mid \cP_0^n),
\end{align*}
where, $0 \leqslant H_\nu(\alpha \mid \cP_0^n) \leqslant H_\nu(\alpha) \leqslant \log 2$ by~\cite[Theorem~4.3(vi)]{Walters82ergth}. Thus, for any partition $\alpha$ of $M$ with $\#\alpha =2$ and any measure $\nu$, 
\begin{align*}
\left| P_\nu(g,\cP_0^n) - P_\nu(g,\cP_0^n \mid \alpha) \right| \leqslant \frac{\log 2}{n}.
\end{align*}

Let $\mu \in \cM(M,T)$. By Lemma~\ref{lemma: structure of measures on singularities}, it follows that $\mu(\cS) = \mu(\bigcup_{n \geqslant 0} \cS_{n})$. Let $n$ be large enough so that $|\mu(\cS) - \mu(\cS_{n+1} \cup \cS_{-1})|$ is less than $\ve$, and less than $ \ve / \sup g$ if $\sup g > 0$.

Let $\delta$ be small enough so that $|\mu(\cS_{n+1} \cup \cS_{-1}) - \mu(\cN_{\delta}(\cS_{n+1} \cup \cS_{-1}))|$ is smaller than $\ve$, and smaller that $\ve / (-\inf g)$ if $\inf g<0$. Notice that, up to decreasing the value of $\delta$, we can assume that $\mu( \partial \cN_{\delta}(\cS_{n+1} \cup \cS_{-1}) \smallsetminus \cS_0) =0$. Now, denote $U = \cN_{\delta}(\cS_{n+1} \cup \cS_{-1})$ and consider the partition $\alpha_{\ve,n} = \{ U , M \smallsetminus U \}$ into two elements. Similarly, let $\delta' < \delta$ and $U' = \cN_{\delta'}(\cS_{n+1} \cup \cS_{-1})$ such that $\mu( \partial U' \smallsetminus \cS_0) =0$ and $ \mu(U \smallsetminus U') < \ve / (n \sup_{M} |g|)$. Finally, let $\varphi \in C^0 (M,\bR)$ be such that $\varphi|_{U^c} \equiv 1$ and $\varphi|_{U'} \equiv 0$.

Therefore, for any $\nu \in \cM(M)$,

\begin{align}\label{eq:decomp_with_cond}
P_\nu(g,\cP_0^n) - P_\mu(g,\cP_0^n) \leqslant \frac{2\log 2}{n} + \left( P_\nu(g,\cP_0^n \mid \alpha_{\ve,n}) - P_\mu(g,\cP_0^n \mid \alpha_{\ve,n} ) \right).
\end{align}
We want to estimate the right hand side. First, notice that since for every $A \in \cP_0^n$, $\partial (A \smallsetminus U) \subset \partial U \smallsetminus \cS_0$ has $\mu$-measure zero, there is a neighbourhood $V_A$ of $\mu$ in $\cM(M)$ such that for any $\nu \in V_A$,
\begin{align*}
\left| \nu (A \smallsetminus U ) - \mu ( A \smallsetminus U ) \right| \leqslant  \frac{\ve}{\# \cP_0^n}.
\end{align*}
Also, since $\mu( \partial U \smallsetminus \cS_0) =0$ and $\cS_0 = \partial M$, there is a neighbourhood $\tilde V_1$ of $\mu$ such that for any $\nu \in \tilde V_1$,
\begin{align*}
\left| \nu(U) - \mu(U)  \right| \leqslant \ve.
\end{align*}
Since $\mu(\partial( M \smallsetminus U)) =0$ and $\mu( \partial (U \smallsetminus U'))=0$ there is a neighbourhood $\tilde V_2$ of $\mu$ such that for any $\nu \in \tilde V_2$,
\begin{align*}
\left| \nu(M \smallsetminus U) - \mu(M \smallsetminus U)  \right| \leqslant \ve, \quad \text{ and } \left| \nu(U \smallsetminus U') - \mu(U \smallsetminus U') \right| \leqslant \frac{\ve}{n \sup |g|}.
\end{align*}
Finally, since $ \varphi \, S_n g$ is continuous, there is a neighbourhood $\tilde V_3$ of $\mu$ such that for any $\nu \in \tilde V_3$,
\begin{align*}
\left| \int \varphi \, S_n g \intd \mu - \int \varphi \, S_n g \intd \nu \right| \leqslant \ve.
\end{align*}
Define $V_{\ve,n} \coloneqq \tilde V_1 \cap \tilde V_2 \cap \tilde V_3 \cap \bigcap_{A \in \cP_0^n} V_A$. Note that $V_{\ve,n}$ is a neighbourhood of $\mu$ in $\cM(M)$.

In the case where $\mu(M \smallsetminus U) > 0$, we will need to consider an even smaller neighbourhood. More precisely, notice that if the value of $\ve$ decreases (as well as $\delta$, accordingly), then $\mu(M \smallsetminus U)$ increases. Thus, if $\mu(M \smallsetminus U_\ve) > 0$ for some $\ve$, there is an $\tilde \ve >0$ such that for every $0 < \ve < \tilde \ve$, $3\ve / \mu(M \smallsetminus U) < 1/2$. Furthermore, there is a neighbourhood $\tilde V_4$ of $\mu$ such that for every $\nu \in \tilde V_4$
\begin{align*}
\left| \frac{\nu(M \smallsetminus U)}{\mu(M \smallsetminus U)} - 1 \right| \leqslant \min \left( \ve,  \frac{\ve}{n \sup_{M} |g|} \right) .
\end{align*}
If $\mu(M \smallsetminus U) > 0$, replace $V_{\ve,n}$ by $V_{\ve,n} \cap \tilde V_4$ and take $\ve < \tilde \ve$. 

\medskip

Expanding the second term of the right-hand-side of \eqref{eq:decomp_with_cond} and assuming $\nu \in V_{\ve,n}$, we get
\begin{align}\label{eq:diff_decomp}
\begin{split}
& P_\nu(g,\cP_0^n \mid \alpha_{\ve,n}) - P_\mu(g,\cP_0^n \mid \alpha_{\ve,n} ) \leqslant \left( \nu(U) P_\nu(g,\cP_0^n \mid U) - \mu(U) P_\mu(g,\cP_0^n \mid U) \right) \\
&\qquad\qquad\qquad\qquad\qquad + \left( \nu(M \smallsetminus U) P_\nu(g,\cP_0^n \mid M \smallsetminus U) - \mu(M \smallsetminus U) P_\mu(g,\cP_0^n \mid M \smallsetminus U) \right).
\end{split}
\end{align}
We begin with the contributions of $U$. Note that
\begin{align*}
\mu(U) P_\mu(g,\cP_0^n \mid U) &\geqslant \mu(U) \int \frac{1}{n} S_n g \intd \mu_U = \int_U \frac{1}{n} S_n g \intd \mu \\
&\geqslant \int_{\cS} \frac{1}{n} S_n g \intd\mu - \int_{\cS \smallsetminus ( \cS_{n+1}\cup \cS_{-1})} \frac{1}{n} S_n g \intd\mu + \int_{U \smallsetminus ( \cS_{n+1}\cup \cS_{-1})} \frac{1}{n} S_n g \intd\mu 
\\
&\geqslant \int_{\cS} g \intd \mu  - 2\ve,
\end{align*}
and,
\begin{align*}
\nu(U) P_\nu(g,\cP_0^n \mid U) \leqslant (\mu(\cS) \pm 3 \ve) \frac{1}{n} \log \sum_{A \in \cP_0^n} \sup_{A} e^{S_n g},
\end{align*}
where the $\pm 3 \ve$ stands for $+3\ve$ if the $\log$ term is positive and $-3\ve$ otherwise. Therefore, the contributions of $U$ in \eqref{eq:diff_decomp} can be bounded by 
\begin{align*}
\nu(U) P_\nu(g,\cP_0^n \mid U) - \mu(U) P_\mu(g,\cP_0^n \mid U) \leqslant (\mu(\cS) \pm 3 \ve)\frac{1}{n} \log \sum_{A \in \cP_0^n} \sup_{A} e^{S_n g} - \int_{\cS} g \intd\mu + 2 \ve.
\end{align*}

We now consider the contributions of $M \smallsetminus U$, that is, the second term in \eqref{eq:diff_decomp}. This contribution can be further split as follows
\begin{align}\label{eq: contribution of M-U}
\begin{split}
\nu(M \smallsetminus U) P_\nu(g,\cP_0^n &\mid M \smallsetminus U) - \mu(M \smallsetminus U) P_\mu(g,\cP_0^n \mid M \smallsetminus U) \\
& = \left( \nu(M \smallsetminus U) - \mu(M \smallsetminus U) \right) P_\mu(g,\cP_0^n \mid M \smallsetminus U) \\
&\quad +  \nu(M \smallsetminus U) \left( P_\nu(g,\cP_0^n \mid M \smallsetminus U) - P_\mu(g,\cP_0^n \mid M \smallsetminus U)  \right).
\end{split}
\end{align}
The first term on the right-hand side can be bounded by
\begin{align*}
\left( \nu(M \smallsetminus U) - \mu(M \smallsetminus U) \right) P_\mu(g,\cP_0^n \mid M \smallsetminus U) &=
\ve \left( \left| \frac{1}{n}\log \sum_{A \in \cP_0^n} \sup_{A} e^{S_n g} \right| + \left| \int \frac{1}{n} S_n g \intd \mu _{U^c} \right| \right) \\
&\leqslant \ve \left( \left| \frac{1}{n}\log \sum_{A \in \cP_0^n} \sup_{A} e^{S_n g} \right| + \sup |g| \right),
\end{align*}
whereas the second one is bounded by
\begin{align*}
\nu(M \smallsetminus U) &\left( P_\nu(g,\cP_0^n \mid M \smallsetminus U) - P_\mu(g,\cP_0^n \mid M \smallsetminus U)  \right) \\
&\leqslant \frac{\nu(M \smallsetminus U)}{n} \sum_{A \in \cP_0^n} \left| - \frac{ \mu(A \smallsetminus U) }{ \mu(M \smallsetminus U) } \log \frac{ \mu(A \smallsetminus U) }{ \mu(M \smallsetminus U) } + \frac{ \nu(A \smallsetminus U) }{ \nu(M \smallsetminus U) } \log \frac{ \nu(A \smallsetminus U) }{ \nu(M \smallsetminus U) } \right| \\
&\qquad + \frac{\nu(M \smallsetminus U)}{n} \left| \int S_n g \intd \mu _{U^c} - \int S_n g \intd \nu_{U^c} \right|.
\end{align*}

Now, since
\begin{align*}
S &\coloneqq \sum_{A \in \cP_0^n} \left| \frac{\mu(A \smallsetminus U)}{\mu(M \smallsetminus U)} - \frac{\nu(A \smallsetminus U)}{\nu(M \smallsetminus U)} \right| \\
&\leqslant \frac{1}{\mu(M \smallsetminus U)} \sum_{A \in \cP_0^n} \left| \left( \frac{\mu(M \smallsetminus U)}{\nu(M \smallsetminus U)} -1 \right) \nu(A \smallsetminus U) \right| + \left| \mu(A \smallsetminus U) - \nu(A \smallsetminus U) \right| \\
&\leqslant \frac{1}{\mu(M \smallsetminus U)} \sum_{A \in \cP_0^n} \left( 2\ve \nu(A \smallsetminus U) + \frac{\ve}{\# \cP_0^n} \right) \leqslant \frac{3 \ve}{\mu(M \smallsetminus U)} < \frac{1}{2},
\end{align*}
we can apply Lemma~\ref{lemma: diff of x log x} and use the Jensen inequality (for the last line) to get that 
\begin{align*}
& \frac{\nu(M \smallsetminus U)}{n} \sum_{A \in \cP_0^n} \left| - \frac{ \mu(A \smallsetminus U) }{ \mu(M \smallsetminus U) } \log \frac{ \mu(A \smallsetminus U) }{ \mu(M \smallsetminus U) } + \frac{ \nu(A \smallsetminus U) }{ \nu(M \smallsetminus U) } \log \frac{ \nu(A \smallsetminus U) }{ \nu(M \smallsetminus U) } \right|\\
& \quad \leqslant \frac{\nu(M \smallsetminus U)}{n} \sum_{A \in \cP_0^n} -\left| \frac{\mu(A \smallsetminus U)}{\mu(M \smallsetminus U)} - \frac{\nu(A \smallsetminus U)}{\nu(M \smallsetminus U)} \right| \log \left| \frac{\mu(A \smallsetminus U)}{\mu(M \smallsetminus U)} - \frac{\nu(A \smallsetminus U)}{\nu(M \smallsetminus U)} \right| \\
&\quad \leqslant \frac{\nu(M \smallsetminus U)}{n} S \left( \log \# \cP_0^n - \log S \right) \leqslant (1+\ve) \frac{3\ve}{n} \log \# \cP_0^n + \frac{e^{-1}}{n} \leqslant \frac{4\ve}{n} \log \# \cP_0^n + \frac{e^{-1}}{n}.
\end{align*}
We now turn to the difference of integrals term.
\begin{align*}
\nu(M \smallsetminus U) &\left| \int S_n g \intd \nu _{U^c} - \int S_n g \intd \mu_{U^c} \right| \\
& \leqslant \left| \int S_n g \, \mathbbm{1}_{U^c} \intd \nu - \int S_n g \, \varphi \intd \nu \right| 
+ \left| \int S_n g \, \varphi \intd \nu - \int S_n g \, \varphi \intd \mu \right| \\
&\quad + \left| \int S_n g \, \varphi \intd \mu - \int S_n g \, \mathbbm{1}_{U^c} \intd \mu \right|
+ \left| \frac{\nu(M \smallsetminus U)}{\mu(M \smallsetminus U)} -1 \right| \int |S_n g| \, \mathbbm{1}_{U^c} \intd \mu \\
&\leqslant n \sup |g| ( \mu(U \smallsetminus U') + \nu(U \smallsetminus U')) + 2\ve \leqslant 5 \ve
\end{align*}

Thus, combining the above estimates inside \eqref{eq: contribution of M-U}, we get that the contribution of $M \smallsetminus U$ is at most
\begin{align*}
\nu(M \smallsetminus U)& P_\nu(g,\cP_0^n \mid M \smallsetminus U) - \mu(M \smallsetminus U) P_\mu(g,\cP_0^n \mid M \smallsetminus U) \\
&\leqslant \ve \left( \left| \frac{1}{n} \log \sum_{A \in \cP_0^n} \sup_{A} e^{S_n g} \right| + \sup |g| + 4 \frac{\log \# \cP_0^n}{n} + \frac{5}{n} \right) + \frac{e^{-1}}{n}.
\end{align*}
Combining the estimates of the contributions of $U$ and $M \smallsetminus U$ inside \eqref{eq:diff_decomp}, we deduce the claim from \eqref{eq:decomp_with_cond}.
\end{proof}

We are now able to prove the main bound of this section.

\begin{corollary}\label{prop:defect of USC}
For every finite horizon Sinai billiard, every $T$-invariant probability measure $\mu$, and every $\cM_0^1$-continuous potential $g$,
\begin{align*}
\limsup_{\nu \rightharpoonup \mu} P_\nu(T,g) \leqslant P_\mu(T,g) + \mu(\cS) \left( P_{\rm top}(T,g) - P_{\mu_\cS}(T,g) \right),
\end{align*}
where, in the $\limsup$, $\nu$ belongs to $\cM(M,T)$ and the convergence is in the weak-$^*$ topology.
\end{corollary}

\begin{proof}
Let $\mu \in \cM(M,T)$ and $\ve > 0$ small enough as in Proposition~\ref{prop: diff of static entropies}. Since $\cP$ is a generating partition, we can choose $n=n(\ve)$ large enough so that 
\begin{align*}
h_\mu(T) \geqslant \frac{1}{n}H_{\mu}(\cP_0^n) - \ve.
\end{align*}
Let $V_{\ve,n}$ be the neighbourhood of $\mu$ given by Proposition~\ref{prop: diff of static entropies} (up to increasing the value of $n$). Therefore
\begin{align*}
P_{\mu}(T,g) &= h_{\mu}(T) + \mu(g) \geqslant \frac{1}{n} H_{\mu}(\cP_0^n) - \ve + \mu(g) = P_{\mu}(g,\cP_0^n) - \ve \\
 &\geqslant \sup_{\nu \in V_{\ve,n}} P_{\nu}(g,\cP_0^n) - \ve - R_{\ve,n}(g) \\
 &\geqslant \sup_{\nu \in V_{\ve,n}} P_{\nu}(T,g) - \ve - R_{\ve,n}(g),
\end{align*}
where we used that $P_{\nu}(T,g) = \inf_n P_{\nu}(g,\cP_0^n)$. We can choose $n=n(\ve)$ such that, taking the limit $\ve \to 0$, we get $n(\ve) \to +\infty$. Therefore, taking the limit as $\ve$ goes to zero of the previous expression yields
\begin{align*}
P_\mu(T,g) \geqslant \limsup_{\nu \rightharpoonup \mu} P_\nu(T,g) - \mu(\cS) P_{\rm top}(T,g) + \int_\cS g \intd \mu \, .
\end{align*}
Finally, since $\mu_\cS$ is atomic, it follows that $h_{\mu_\cS}(T)=0$. Hence 
\begin{align*}
P_{\mu_\cS}(T,g) = \int g \intd \mu_\cS = \frac{1}{\mu(\cS)} \int_\cS g \intd \mu,
\end{align*}
which finishes the proof.
\end{proof}

\subsection{Equilibrium States for some potentials}

In this part we prove that if $P_{\rm top}(T,g) - \sup g > 0$, then $g$ admits equilibrium states. The following lemma, combined with the assumption on $g$, will serve to discard measures from $\cM_{\cS}(M,T)$.

\begin{lemma}\label{lemma: ngbh of singular measures with small entropy}
For all $\mu \in \cM_{\cS}(M,T)$ and $\ve >0$, there exists $n_\mu \geqslant 1$ and an open neighbourhood $U_\mu \subset \cM(M)$ of $\mu$ such that for all $n \geqslant n_\mu$ and $\nu \in U_\mu$, $\frac{1}{n} H_\nu(\cP_0^n) \leqslant \ve$. 

In particular $\sup\{ h_{\nu}(T) \mid \nu \in U_\mu \cap \cM(M,T) \} \leqslant \ve$.
\end{lemma}

Note that measures arbitrarily close to $\cM_{\cS}(M,T)$ of positive entropy have been constructed in~\cite{CDLZ24}.

\begin{proof}
Let $\mu \in \cM_{\cS}(M,T)$. Then, by Lemma~\ref{lemma: structure of measures on singularities}, $\mu = \sum_{i \in I} \omega_i \delta_{\cO(x_i)}$, where $I$ is at most countable. We first consider the case where $I$ is infinite. We identify $I$ with $\bN$.

Let $\ve > 0$ and $i_0$ large enough so that $\sum_{i > i_0} \omega_i < \ve$. Since $\cO_\mu \coloneqq \bigcup_{i=1}^{i_0} \cO(x_i) \subset M$ is a finite set and $\lim_{n \to \infty} \diam( \cP_{-n}^n) = 0$, there exists $n_0$ such that for all $n \geqslant n_0$ and $A \in \cP_{-n}^n$, $\# \bar A \cap \cO_\mu \leqslant 1$. Notice that, writing $C_\mu = \# \cO_\mu$, there are at most $2Kn C_\mu$ elements $A$ of $\cP_{-n}^n$ where the equality holds. Define $U_n$ the interior of the union of those elements of $\cP_{-n}^n$. For any $x \in \cO_\mu$, we write $U_n(x)$ the connected component of $U_n$ containing $x$. Up to increasing the value of $n$, we can assume that the $U_n(x)$ are pairwise disjoint.

Since each $U_n(x)$, $x \in \cO_\mu$, is open, there is an open neighbourhood $W_{\ve,n}$ of $\mu$ in $\cM(M)$ such that for every $\nu \in W_{\ve,n}$, $\nu(U(x)) > \mu(U(x)) - \ve / C_\mu$. In particular, by choice of $i_0$, $1 - \nu(U_n) \leqslant 2 \ve$. Therefore, for any $\nu \in W_{\ve,n}$,
\begin{align*}
\frac{1}{n}H_\nu(\cP_0^n) &\leqslant \frac{1}{n}H_\nu(\cP_{-n}^n) = \frac{1}{n}\nu(U_n)H_\nu(\cP_{-n}^n \mid U_n) + \frac{1}{n}(1-\nu(U_n))H_\nu(\cP_{-n}^n \mid U_n^c) \\
&\leqslant \frac{\log 2Kn C_\mu}{n} + 2\ve \frac{\log \cP_{-n}^n}{n} \leqslant 5 \ve \htop (T),
\end{align*}
where the last inequality holds as long as $n$ is large enough, say $n_\mu$. The first part of the lemma follows by setting $U_\mu = W_{\ve/5 \htop (T) ,n}$ for some large enough $n$. The case where $I$ is finite is analogous.

The last part follows from that $h_\nu(T) \leqslant \frac{1}{n}H_\nu(\cP_0^n)$ for any $T$-invariant measure $\nu$.
\end{proof}

\begin{corollary}\label{corol: existence FH}
Let $g$ be a $\cM_0^1$-continuous potential such that $ P_{\rm top}(T,g) - \sup g > 0 $. Then there exist equilibrium states associated to $g$ for the collision map. Furthermore, the variational principle holds.
\end{corollary}

The proof is very similar in spirit to the one of Proposition~\ref{prop: var principle}, but more technical. The additional assumption on $g$ ensure that the limit measure does not belong to $\cM_{\cS}(M,T)$.

\begin{proof}
Denote $g_q = \frac{1}{q}S_q g$ the $q$-th Birkhoff average of $g$, and $h_q = \frac{1}{q}\sum_{k=0}^{q-1}S_k g$ the function such that $g_q - g = h_q \circ T - h_q$. First, note that for any sequences $q_j$ and $n_j$ tending to infinity such that $q_j/n_j$ goes to zero, we get
\begin{align}\label{eq: average and pressure}
\lim_{j \to \infty} \frac{1}{n_j} \log \sum_{A \in \cP_0^{n_j}} \sup_{A} e^{S_{n_j} g_{q_j}} = P_{\rm top}(T,g).
\end{align} 
Indeed, remark that $S_n g_q = S_n g + h_q \circ T^n - h_q$ and $\sup | h_q \circ T^n - h_q | \leqslant \sup h_q - \inf h_q \leqslant \frac{q-1}{2}(\sup g - \inf g)$. Thus \eqref{eq: average and pressure} holds from the assumption $q_j/n_j \to 0$.

Define $\nu_{n,q}$ and $\mu _{n,q}$ as in the proof of Proposition~\ref{prop: var principle}, but for the potential $g_q$ instead of $g$. For any fixed $q$, up to extracting, $\mu _{n,q}$ converges to some $\mu _{q}$ as $n \to \infty$. Also, up to extracting, $\mu_{q!}$ converges to some $\mu$. Denote $q_i$ the sequence such that $\mu _{q_i !}$ converges to $\mu$. By Proposition~\ref{prop: averaging_measures}, $\mu$ and every $\mu_q$ are $T$-invariant.

Let $\ve > 0$ and fix some $q_* \geqslant 1$ belonging to the sequence $q_i$. Let $V_{\ve,q_*}$ the open neighbourhood of $\mu $ provided by Proposition~\ref{prop: diff of static entropies} with respect to the potential $g$. For $i$ large enough, $\ell_i = \frac{q_i!}{q_*}$ is a sequence of integers and $\mu _{\ell_i q_*} \in V_{\ve,q_*}$. Now, let $n_i$ be large enough so that $\mu _{n_i,\ell_i q_*} \in V_{\ve,q_*}$. We can assume that $\ell_i / n_i$ converges to $0$.

Thus, \eqref{eq: variational_principle} reads for the measure $\mu _{n_i,\ell_i q_*}$ and the potential $g_{\ell_i q_*}$,
\begin{align}\label{eq: main bound for existence}
\begin{split}
\frac{\log \chi_{n_i}}{n_i} + \frac{1}{n_i} \log \sum_{A \in \cP_0^{n_i}} \sup_{A} e^{S_{n_i} g_{\ell_i q_*}} 
&\leqslant \frac{2 \ell_i q_*}{n_i} \log \# \cP + \frac{1}{\ell_i q_*}H_{\mu _{n_i,\ell_i q_*}}(\cP_0^{\ell_i q_*}) + \int g_{\ell_i q_*} \intd \mu _{n_i,\ell_i q_*} \\
&\leqslant \frac{2\ell_i q_*}{n_i} \log \# \cP + P_{\mu _{n_i,\ell_i q_*}}(g,\cP_0^{\ell_i q_*}) \\
&\leqslant \frac{2\ell_i q_*}{n_i} \log \# \cP + P_{\mu _{n_i,\ell_i q_*}}(g,\cP_0^{q_*}),
\end{split}
\end{align} 
where in the last line we used the sub-additivity of $n \mapsto H_\nu(\cP_0^n)$. Now, since $\mu _{n_i,\ell_i q_*} \in V_{\ve,q_*}$, we can apply Proposition~\ref{prop: diff of static entropies}, and then take the limit $i \to + \infty$. Thus
\begin{align*}
P_{\rm top}(T,g) \leqslant P_{\mu}(g,\cP_0^{q_*}) + \frac{\mu(\cS)}{q_*}\sum_{A \in \cP_0^{q_*}} \sup_{A} e^{S_{q_*}g} - \int_{\cS} g \intd \mu + R_{\ve,q_*}(g).
\end{align*}
Since this holds for every $\ve > 0$ and every $q_*$ belonging to the sequence $q_i \to +\infty$, we can take the limits $\ve \to 0$ and $q_* \to + \infty$, which yield
\begin{align*}
(1- \mu(\cS)) P_{\rm top}(T,g) \leqslant (1 - \mu(\cS)) P_{\mu _{\cS^c}}(T,g).
\end{align*} 
In particular, either $\mu = \mu _{\cS}$, or $\mu _{\cS^c}$ is an equilibrium state which satisfies the variational principle. 

By contradiction, assume that $\mu = \mu _{\cS} \in \cM(M,T)$. Let $\ve$ be small enough so that $P_{\rm top}(T,g) > 4\ve + \sup g$. Let $U_\mu$  and $n_\mu$ be as provided as in Lemma~\ref{lemma: ngbh of singular measures with small entropy} with respect to $\mu$ and $\ve$. Fix some $q_* > n_\mu$. Let $i$ be large enough so that the sequence in \eqref{eq: average and pressure} is $\ve$ close to its limit, and also so that 
\begin{align*}
0 < \frac{2\ell_i q_*}{n_i} \log \# \cP \leqslant \ve \, , \quad \text{ and} \quad 0 < - \frac{\log \chi_{n_i}}{n_i} < \ve .
\end{align*}
Furthermore, for large enough $i$, we get that $\mu _{\ell_i q_*} \in U_\mu$. Thus, up to increasing the value of $n_i$, we also get that $\mu _{n_i,\ell_i q_*} \in U_\mu$. Therefore, from \eqref{eq: main bound for existence}, it follows that
\begin{align*}
-2 \ve + P_{\rm top}(T,g) \leqslant 2\ve + \sup g, 
\end{align*}
hence a contradiction.
\end{proof}

\begin{corollary}\label{corol: max sequence}
Let $g$ be a $\cM_0^1$-continuous potential such that $ P_{\rm top}(T,g) - \sup g > 0 $. If there exists a sequence of measure $\mu_k \in \cM(M,T)$ such that 
\[ P_{\rm top}(T,g) = \lim_{k \to \infty} P_{\mu_k}(T,g), \]
then for any accumulation point $\mu$, the measure $\mu_{\cS^c}$ is an equilibrium state for $g$ and $T$.
\end{corollary}

\begin{proof}
Let $g$ and $\mu_k$ as in the statement of the corollary. By compactness, we can assume that $\mu_k$ converges to some invariant measure $\mu$. Therefore, by Corollary~\ref{prop:defect of USC},
\begin{align*}
P_{\rm top}(T,g) = \lim_{k \to \infty} P_{\mu_k}(T,g) \leqslant P_{\mu}(T,g) + \mu(\cS) ( P_{\rm top}(T,g) - P_{\mu_\cS}(T,g)).
\end{align*}
Now, using that the Kolmogorov--Sinai entropy is harmonic, it follows that 
\begin{align*}
P_{\mu}(T,g) = (1 - \mu(\cS)) P_{\mu _{\cS^c}}(T,g) + \mu(\cS) P_{\mu_\cS}(T,g).
\end{align*}
Thus,
\begin{align*}
(1 - \mu(\cS)) P_{\rm top}(T,g) \leqslant (1 - \mu(\cS)) P_{\mu_{\cS^c}}(T,g).
\end{align*}
In particular, either $\mu(\cS)=1$ (and $\mu _{\cS^c}$ is not defined) or $P_{\mu _{\cS^c}}(T,g) = P_{\rm top}(T,g)$.

By contradiction, assume that $\mu(\cS)=1$. Let $\ve > 0$ be small enough so that $ P_{\rm top}(T,g) - \sup g > 2 \ve $. By Lemma~\ref{lemma: ngbh of singular measures with small entropy}, for any $k$ large enough $h_{\mu_k}(T) < \ve$. Hence
\begin{align*}
P_{\mu_k}(T,g) \leqslant \ve + \sup g < P_{\rm top}(T,g) - \ve,
\end{align*}
which contradicts the assumption on the sequence $\mu_k$.
\end{proof}

\begin{corollary}\label{corol:USC_at_ergodic_measures}
The Kolmogorov--Sinai entropy is upper semi-continuous at every ergodic measure.
\end{corollary}

\begin{proof}
Let $\mu$ be an ergodic invariant measure. Since $\cS$ is invariant, either $\mu(\cS)=0$ or $\mu(\cS^c)=0$. In the first case, it follows from Corollary~\ref{prop:defect of USC} that the Kolmogorov--Sinai entropy is upper semi-continuous at $\mu$. In the second case, it must be that $\mu \in \cM_{\cS}(M,T)$. By Lemma~\ref{lemma: structure of measures on singularities}, since $\mu$ is ergodic, it is supported by a single periodic orbit. Hence $h_{\mu}(T)=0$. The upper-semi continuity of the Kolomogorov--Sinai entropy at $\mu$ follows from Lemma~\ref{lemma: ngbh of singular measures with small entropy}.
\end{proof}

\subsection{Upper semi-continuity of the Kolmogorov--Sinai entropy}

Remark that in the case where $g \equiv 0$, and $\mu \in \cM_{\cS}(M,T)$, the bound provided by Corollary~\ref{prop:defect of USC} reduces to a trivial estimate. On the other hand, according to Lemma~\ref{lemma: ngbh of singular measures with small entropy}, this bound should vanish.

Nonetheless, combining ideas from the proofs of Proposition~\ref{prop: diff of static entropies} and Lemma~\ref{lemma: ngbh of singular measures with small entropy}, we can deduce the upper semi-continuity of the metric entropy, thus improving Corollary~\ref{corol:USC_at_ergodic_measures}.
\begin{theorem}\label{thm: USC KS entropy}
For any finite horizon Sinai billiard, the Kolmogorov--Sinai entropy of the collision map is upper semi-continuous.
\end{theorem}
This result follows from the next estimate on static entropies.
\begin{lemma}\label{lemma: USC-ish}
For every $\mu \in \cM(M,T)$, every $\cM_0^1$-continuous potential $g : M \to \mathbbm{R}$, every small enough $\ve > 0$ and every large enough $n$, there is a neighbourhood $V'_{\ve,n} \subset \cM(M,T)$ of $\mu$ such that for every $\nu \in V'_{\ve,n}$,
\begin{align*}
P_{\nu}(g,\cP_{-n}^n) - P_{\mu}(g,\cP_{-n}^n) \leqslant \mu(\cS) \left( \sup g - \int g \intd \mu_\cS \right) + R'_{\ve,n}(g),
\end{align*}
where $R'_{\ve,n}(g)$ goes to zero as $\ve \to 0$ and $n \to \infty$.
\end{lemma}
The proof relies on a notion of conditional static pressure slightly different from the one used in the proof of Proposition~\ref{prop: diff of static entropies}. Namely, we define
\begin{align*}
P_\mu(g,\cP_{-n}^n \mid \cC) &\coloneqq \sum_{C \in \cC} \mu(C) P_\mu(g,\cP_{-n}^n \mid C), \text{ where } \\
P_\mu(g,\cP_{-n}^n \mid C) &\coloneqq \frac{1}{2n} \left( H_{\mu_C}(\cP_{-n}^n) + \int S_{2n}g \circ T^{-n} \intd \mu_C \right).
\end{align*}

\begin{proof}[Proof of Lemma~\ref{lemma: USC-ish}]
The proof is an adaptation of the one from Proposition~\ref{prop: diff of static entropies}, with a different conditioning. We start by describing the construction the new partition $\alpha_{\ve,n}$. 

Let $\mu \in \cM(M,T)$. Assume first that $\mu(\cS) > 0$ (otherwise, we can just apply Corollary~\ref{prop:defect of USC}). According to Lemma~\ref{lemma: structure of measures on singularities}, we can decompose $\mu_\cS$ into a convex combination of measure supported by grazing periodic orbit $\mu_\cS = \sum_{i \geqslant 0} \omega_i \delta_{\cO(x_i)}$. Let $\ve > 0$ and $i_0 \geqslant 0$ be such that $\sum_{i > i_0} \omega_i < \ve$. For future convenience, denote $\cO_\mu = \bigcup_{i=0}^{i_0} \cO(x_i)$ and $C_\mu = \# \cO_\mu$. Let $n_0$ be large enough so that $\cO_\mu \subset \cS_{-n_0} \cap \cS_{n_0}$ and for every $n \geqslant n_0$, every $A \in \cP_{-n}^n$, $\bar A$ contains at most one point of $\cO_\mu$. Notice that in particular, for any $n \geqslant n_0$, $|\mu(\cS) - \mu(\cS_n \cup \cS_{-n})| < \ve$.

Define $\cU_n = \{ A \in \cP_{-n}^n \mid \bar A \cap \cO_\mu \neq \emptyset \}$. Thus $\# \cU_n \leqslant 2Kn C_\mu$. Let $\delta > 0$ be small enough so that $|\mu(\cS_n \cup \cS_{-n}) - \mu(\cN_{\delta}(\cS_n \cup \cS_{-n})| < \ve$. Up to decreasing the value of $\delta$, we can assume that $\mu(\partial\cN_{\delta}(\cS_n \cup \cS_{-n} \smallsetminus \cS_0 ) =0$, as well as, for every $x \in \cO_\mu$, $B(x,\delta) \smallsetminus \cS_0 \subset \mathrm{Int} \bigcup \bar A$ (where the union ranges over the $A$ in $\cU_n$ with $x \in \bar A$), and $\mu(\partial B(x,\delta) \smallsetminus \cS_0) =0$.
For ease of notation, denote $U = \cN_\delta(\cS_{-n} \cup \cS_{n})$ and $B_x = B(x,\delta)$. We can then consider the following family of finite partitions of $M$
\begin{align*}
\alpha'_{\ve,n} \coloneqq \{ M \smallsetminus U, U_{\rm reg}, (B_x)_{x \in \cO_\mu} \} \text{ where } U_{\rm reg} = U  \smallsetminus \bigsqcup_{x \in \cO_\mu} B_x.
\end{align*}
Notice that $\#\alpha'_{\ve,n} \leqslant 2Kn C\mu + 2$ and that for every $A \in \alpha'_{\ve,n}$, $\mu(\partial A \smallsetminus \cS_0)=0$.
Exactly as in the proof of Proposition~\ref{prop: diff of static entropies}, define $\delta' > 0$, $U' \subset U$ and $\varphi : M \to \bR$. Define the neighbourhood $V'_{\ve,n}$ of $\mu$ by $V_{\ve,n} \cap \bigcap_{x \in \cO_\mu} V'_x$, where $V'_x$ is open and such that for every $\nu \in V'_x$, $|\mu(B_x) - \nu(B_x)| < \ve$.

We can now estimate the contributions to the pressure jump of each element of $\alpha'_{\ve,n}$. The computation of the contribution of $U^c$ is analogous to Proposition~\ref{prop: diff of static entropies}. More explicitly, we have 
\begin{align*}
&\nu(U^c)P_\nu(g,\cP_{-n}^n \mid U^c) - \mu(U^c)P_\mu(g,\cP_{-n}^n \mid U^c) \\
&\leqslant \frac{\ve}{2n}\left(\left| \log \sum_{A \in \cP_{-n}^n} \sup_{A} e^{S_{2n}g \circ T^{-n}} \right| + 4 \log \# \cP_{-n}^n + 5 \right) + \ve \sup|g| + \frac{e^{-1}}{2n}
\end{align*}
The contribution of $U_{\rm reg}$ is obtain by noticing that
\begin{align*}
\mu(U_{\rm reg}) P_{\mu}(g,\cP_{-n}^n \mid U_{\rm reg}) &\geqslant \int_{U_{\rm reg}} \frac{1}{2n} S_{2n}g \circ T^{-n} \intd \mu, \\
\nu(U_{\rm reg}) P_{\nu}(g,\cP_{-n}^n \mid U_{\rm reg}) &\leqslant \frac{\mu(U_{\rm reg}) \pm 3\ve}{2n} \log \sum_{A \in \cP_{-n}^n} \sup_{A} e^{S_{2n}g \circ T^{-n}},
\end{align*}
where the $\pm$ has the same sign as the $\log$.

Finally, the contributions of each $B_x$ is as follows.
\begin{align*}
\mu(B_x)  P_{\mu}(g,\cP_{-n}^n \mid B_x) &\geqslant \int_{B_x} \frac{1}{2n} S_{2n}g \circ T^{-n} \intd \mu, \\
\nu(B_x)  P_{\nu}(g,\cP_{-n}^n \mid B_x) &\leqslant \frac{\nu(B_x)}{2n} \sum_{A \in \cP_{-n}^n} \nu_{B_x}\left( -\log \nu(B_x) + \sup_{B_x} S_{2n}g \circ T^{-n}\right) \\
&\leqslant \frac{\nu(B_x)}{2n} \log \sum_{A \cap B_x \neq \emptyset} \sup_{A} e^{S_{2n}g \circ T^{-n}} \leqslant \frac{\nu(B_x)}{2n} \log \left( 2n K C_\mu e^{2n \sup g} \right)\\
&\leqslant (\mu(B_x) \pm \ve) \left( \frac{\log 2n K C_\mu}{2n} + \sup g \right),
\end{align*}
where the $\pm$ has the sign of the rightmost parenthesis.

Combining all these estimates yields the claimed bound.
\end{proof}
\begin{proof}[Proof of Theorem~\ref{thm: USC KS entropy}]
For $g \equiv 0$, $\mu \in \cM(M,T)$ and $\ve > 0$ we have that for any $n \geqslant n_\ve$,
\begin{align*}
h_\mu(T) \geqslant H_\mu(\cP_{-n}^n) - \ve &\geqslant \sup_{\nu \in V'_{\ve,n}} H_{\nu}(\cP_{-n}^n) - \ve - R'_{\ve,n}(0) \geqslant \sup_{\nu \in V'_{\ve,n}} h_\nu(T) - \ve - R'_{\ve,n}(0)
\end{align*}
Taking the limits $\ve \to 0$ and $n_\ve \to + \infty$ yields the result.
\end{proof}

\subsection{Distribution of periodic orbits}

Before focusing on the distribution of periodic orbits, we give bounds on sums related to the periodic points. This is the content of the following lemma\footnote{The estimate \eqref{eq: counting per orbits} is due to Mark Demers. The author is grateful to him for allowing to use his proof.}\footnote{Equation \eqref{eq: counting per orbits sharper} was initially stated for $g \equiv 0$. The author is grateful to Yuri Lima for his explanations on how to generalise it.}.

\begin{lemma}\label{lemma:number_of_periodic_points}
For all $n \geqslant 2$ and $A \in \cP_0^n$, $A$ contains at most one periodic orbit of period $n$. In particular, for any $\cM_0^1$-continuous potential $g$,
\begin{align}\label{eq: counting per orbits}
\sum_{x \in \Fix T^n} e^{S_n g(x)} \leqslant \sum_{A \in \cP_0^n} \sup_{A} e^{S_n g}.
\end{align} 
Assuming that $g$ is piecewise H\"older and has sparse recurrence, that is, if $P_{\rm top}(T,g) - \sup g > s_0 \log 2$, there is a constant $C > 0$ such that for every $n \geqslant 2$,
\begin{align}\label{eq: counting per orbits sharper}
C^{-1} e^{n P_{\rm top}(T,g)} \leqslant \sum_{x \in \Fix T^n} e^{S_n g(x)} \leqslant C e^{n P_{\rm top}(T,g)}
\end{align}
\end{lemma}

This lemma gives a positive answer to~\cite[Problem~5]{gutkin2012}.

\begin{proof}
Let $n \geqslant 2$, $A \in \cP_0^n$ and $x,y \in \Fix T^n \cap A$. Since $T^n$ is continuous on $A$, we get that $T^n x= x$ and $T^n y = y$ both belong to $T^n A \in \cP_{-n}^0$. Since $A \cap T^n A \in \cP_{-n}^n$, applying $T^{-n}$ we get that $x$ and $y$ belong to the same element of $\cP_0^{2n}$. Repeating this argument, we obtain that $x$ and $y$ belong to the same element of $\cP_{-2^k n}^{2^k n}$ for every $k \geqslant 0$. Since $\diam \cP_{-k}^k$ goes to zero as $k$ goes to infinity, it must be that $x=y$. From this, \eqref{eq: counting per orbits} follows.

Now, in the case $g \equiv 0$, from sparse recurrence it follows from~\cite[Proposition~4.6]{BD2020MME} that there is a constant $C >0$ such that $\# \cP_0^n \leqslant C e^{n \htop (T)}$. The lower estimate on $\# \Fix T^n$ is the content of~\cite[Theorem~1.5]{Bu} (itself bootstrapping from~\cite[Corollary~2.7]{BD2020MME} and~\cite[Theorem~1.3]{LM}). 

For a more general potential, as in the statement of the lemma, a similar strategy can be applied. Indeed, the upper estimate is the content of~\cite[Proposition~3.10]{Ca1} (note that sparse recurrence implies the \emph{small singular pressure} assumptions, as explained in~\cite{BCD}). For the lower estimate, we also rely on the coding $\pi : (\Sigma, \sigma) \to (M,T)$ constructed in~\cite[Theorem~1.3]{LM} with hyperbolicity parameter related to the unique equilibrium state $\mu_g$ associated to $g$.  

According to~\cite[Corollary~6.3]{Ca1}, $\mu_g$ is $T$-adapted in the sense of~\cite{LM}. Since it is $\chi$-hyperbolic from the choice of $\chi$,~\cite[Theorem~1.3]{LM} implies the existence of $\hat \mu_g \in \cM_{\rm erg}(\Sigma^\#,\sigma)$ such that $\pi_* (\hat \mu_g) = \mu_g$. Since $\pi_{| \Sigma^\#}$ is finite-to-one, $\pi_*$ preserves the entropy, and thus the metric pressure. Therefore, $\hat \mu_g$ is an equilibrium state for $\Sigma$. 

From~\cite[Theorem~4.6]{sarig2009}, since the potential $\hat g \coloneqq g \circ \pi$ admits $\hat \mu_g$ as an equilibrium state, it is positively recurrent. Thus, by the generalised Ruelle's Perron--Frobenius theorem (see~\cite[Theorem~3.4]{sarig2009}), we get that there is a cylinder $[a]$ such that the following convergence holds,
\begin{align*}
\lambda^{-n} L_{\hat g}^n \mathbbm{1}_{[a]} \to c > 0,
\end{align*}
where $L_{\hat g}$ is the weighted transfer operator associated to $\sigma$ and $\hat g$, and $\lambda = \exp ( P_G(\hat g) )$ where $P_G(\hat g)$ is the Gurevich pressure of $\hat g$. Note that by the existence of $\hat \mu_g$, $P_G(\hat g) \geqslant P_{\rm top}(T,g)$.
Now, since 
\begin{align*}
(\mathbbm{1}_{[a]}L_{\hat g}^n \mathbbm{1}_{[a]}) (X) = \sum_{Y \in \sigma^{-n}(X)} e^{S_n \hat g (Y)},
\end{align*}
where in the sum $Y$ ranges over all the sequences of the form $(y_0, \ldots, y_{n-1},X)$, with $y_0=y_n=a$. Since $g$ and $\pi$ are both H\"older, $\hat g$ is also H\"older. Thus, there is a constant $K > 0$ such that $|S_n \hat g (Y) - S_n \hat g (Z_Y) | \leqslant K$, where $Z_Y$ is the periodic orbit of period $n$ whose $n$-th first coordinates coincide with the ones of $Y$. Thus,
\begin{align}\label{eq: sum Z}
\sum_{Z \in [a] \cap \Fix \sigma^n} e^{S_n \hat g (Z)} \geqslant \frac{c}{2} e^{-K} e^{n P_G(\hat g)}
\end{align}
for every $n$ large enough. Up to decreasing the value of $c$, the claimed result holds for every $n \geqslant 2$.
\end{proof}

In this subsection, given a potential $g$, we consider the measures supported on periodic orbits, denoted
\begin{align*}
\mu_n = \frac{1}{\sum_{x \in \Fix T^n} e^{S_n g(x)} } \sum_{x \in \Fix T^n} e^{S_n g(x)} \delta_x.
\end{align*}

\begin{corollary}\label{corol: equidist of per orbits}
Assume that sparse recurrence holds for a $\cM_0^1$-piecewise H\"older potential $g$ and let $\mu$ be an accumulation point of the sequence $\mu_n$. Then $\mu_{\cS^c} = \mu_g$, where $\mu_{\cS^c} \in \cM(M,T)$ is the conditional measure of $\mu$ on $\cS^c$

Further assuming that the singularities are negligible, then $\mu_n \rightharpoonup \mu_g$ as $n$ goes to infinity. 
In particular, when $g \equiv 0$, the periodic orbits equidistribute with respect to the unique measure of maximal entropy.
\end{corollary}

\begin{proof}
From Lemma~\ref{lemma:number_of_periodic_points}, it follows that for every $A \in \cP_0^n$, $\mu_n(A)$ is equal either to $0$ or to $1/ \sum_{x \in \Fix T^n} e^{S_n g(x)}$. Therefore, assuming sparse recurrence,
\begin{align}\label{eq: pressure per pts part 1}
\begin{split}
\frac{1}{n} H_{\mu_n}(\cP_0^n) + \int g \intd \mu_n &= - \frac{1}{n}\sum_{x \in \Fix T^n} \mu_n(\{x \}) \left( \log \mu_n(\{ x \}) - S_n g(x) \right) \\
&= \frac{1}{n}\log \sum_{x \in \Fix T^n} e^{S_n g(x)} \geqslant P_{\rm top}(T,g) - \frac{\log C}{n}.
\end{split}
\end{align}
For any $q \leqslant n$, write the Euclidean division of $n$ by $q$, $n=j_0 q + r_0$. Now, since $\cP_0^n = \bigvee_{j=0}^{j_0-1} T^{-jq} \cP_0^{q} \vee \bigvee_{i=j_0 q +1}^n T^{-i} \cP$, we get by sub-additivity and $T$-invariance of $\mu_n$,
\begin{align}\label{eq: pressure per pts part 2}
\frac{1}{n} H_{\mu_n}(\cP_0^n) &\leqslant \frac{j_0}{n} H_{\mu_n}(\cP_0^q) +  \frac{r_0}{n} H_{\mu_n}(\cP) \leqslant \frac{1}{q} H_{\mu_n}(\cP_0^q) + \frac{q}{n} \log \# \cP.
\end{align}

Consider $n_k$ and $\mu$ such that $\mu_{n_k} \rightharpoonup \mu$ as $n_k$ goes to infinity. Let $\ve > 0$. Up to extracting again, we can assume that each $\mu _{n_k}$ belong to the neighbourhood $V_{\ve,n_k}$ of $\mu$ from Proposition~\ref{prop: diff of static entropies}. Therefore, combining the two above estimates, taking the limit $n_k \to \infty$ and applying Proposition~\ref{prop: diff of static entropies}, we get that
\begin{align*}
P_{\rm top}(T,g) \leqslant P_\mu(g,\cP_0^q) + \frac{\mu(\cS)}{q} \log \sum_{A \in \cP_0^q} \sup_{A} e^{S_q g} - \int_{\cS} g \intd \mu + R_{\ve,q}(g).
\end{align*}
Now, taking successively the limits $q \to \infty$ and $\ve \to 0$, and writing $\mu = (1-\mu(\cS))\mu _{\cS^c} + \mu(\cS)\mu _{S}$ yield
\begin{align*}
P_{\rm top}(T,g) \leqslant (1-\mu(\cS)) P_{\mu _{ \cS^c}}(T,g) + \mu(\cS) P_{\rm top}(T,g),
\end{align*}
where we used that, since $\mu _{ \cS} \in \cM_{\cS}(M,T)$, $h_{\mu _{ \cS}}(T) = 0$. Therefore
\begin{align*}
(1-\mu(\cS)) P_{\rm top}(T,g) = (1-\mu(\cS)) P_{\mu _{ \cS^c}}(T,g).
\end{align*}
Finally, if $\mu(\cS) = 1$, then $\mu = \mu _{ \cS}$. Otherwise, $\mu _{ \cS^c}$ has maximal pressure. By~\cite[Theorem~2.4]{BD2020MME}, the equilibrium state is unique. Therefore $\mu _{ \cS^c} = \mu_g$, where $\mu_g$ is the unique equilibrium state of $T$ and $g$.

To finish the proof, we show that $\mu(\cS) \neq 1$. By contradiction, assume that $\mu = \mu _{\cS} \in \cM_{\cS}(M,T)$. Since $g$ satisfies the sparse recurrence condition, we can choose $\ve$ small enough so that $P_{\rm top}(T,g) - \sup g > 3 \ve$. Combining \eqref{eq: pressure per pts part 1} with \eqref{eq: pressure per pts part 2}, we get that for every $n$
\begin{align*}
P_{\rm top}(T,g) - \frac{\log C}{n} \leqslant \frac{q \log \# \cP}{n} + P_{\mu_n}(g,\cP_0^q).
\end{align*}
Thus, for $q \geqslant n_\mu$ and every $n$ large enough so that $\mu_n \in U_\mu$, where $n_\mu$ and $U_\mu$ are provided by Lemma~\ref{lemma: ngbh of singular measures with small entropy}, up to increasing the value of $n$, we have
\begin{align*}
P_{\rm top}(T,g) - \ve \leqslant 2 \ve + \sup g,
\end{align*}  
hence a contradiction. Therefore $\mu (\cS) < 1$.

Finally, assuming that singularities are negligible, we get that $\cM_{\cS}(M,T)$ is empty. Hence $\mu_n$ converges to the unique equilibrium state $\mu_g$.
\end{proof}

\section{Topological tail entropy}\label{sect:tail_entropy}

Before stating the results from this section, we recall the formal definition of the topological tail entropy (as in~\cite{burguet}). Let $d(x,y)$ be the Euclidean metric between $x$ and $y$ if they belong to the same connected component $M_i$ of $M = \bigsqcup_{i=1}^D M_i$, and otherwise fix $d(x,y) = 10 \max_{1 \leqslant i \leqslant D} \diam(M_i)$. Denote $d_n$ the $n$-th Bowen distance, that is 
\begin{align*}
d_n(x,y) \coloneqq \max_{0 \leqslant i < n} d(T^i x, T^i y),
\end{align*} 
and $B(x,\ve,n)$ the open ball centred at $x$ of radius $\ve$ for the $d_n$ metric. A set $X \subset Y$ (where $Y \subset M$) is said to be a $(n, \delta )$-spanning subset of $Y$ if $Y \subset \bigcup_{x \in X} B(x,\delta,n)$ and a $(n, \delta)$-separated subset of $Y$ if for all distinct $x, y \in X$, $d_n(x,y) \geqslant \delta$. We use the notation $r(n, \delta , Y )$ (respectively, $ s(n, \delta, Y)$) for the minimal cardinality of a $(n, \delta)$-spanning subset of a set $Y \subset M$ (respectively, the maximal cardinality of a $(n, \delta )$-separated subset of $Y$). Now, define the topological tail entropy to be 
\begin{align}\label{eq:topo_tail_entropy}
h^*(T) \coloneqq \lim_{\ve \to 0} \lim_{\delta \to 0} \limsup_{n \to +\infty} \frac{1}{n} \sup_{x \in M} \log r(n,\delta, B(x,\ve,n)).
\end{align}
Noticing that in the case $T$ is a continuous map, replacing $B(x,\ve,n)$ by $M$ in \eqref{eq:topo_tail_entropy} gives back one of Bowen's definition of topological entropy, the topological tail entropy can therefore be interpreted as a rate of local creation of entropy, at the refining scales given by Bowen's balls.

\medskip

In this section, even though the relation between the upper semi-continuity of the entropy map and the topological tail entropy is not clear in the case of transformations with discontinuities, such as the billiard map, we provide an estimate on the topological tail entropy $h^*(T)$ in Corollary~\ref{prop: top tail entropy bound}. To do so, we split the contributions of points in $M \smallsetminus \cS_n$ and $\cS_n$ in Propositions~\ref{lemma: regular points} and \ref{lemma: singular points} respectively. 

The statements of the two propositions rely on two crucial quantities, $\ve_0$ and $\Lambda$ we recall now. Baladi and Demers~\cite[Lemma~3.4]{BD2020MME} proved that\footnote{The definition of the Bowen metric $d_n$ in~\cite{BD2020MME} corresponds here to $d_{n+1}$.} there exists $\ve_0 > 0 $ with the property that for any $n \geqslant 0$, if $x,y$ lie in different elements of $\cM_0^n$, then $d_{n+1}(x,y) \geqslant \ve_0$. The collision map enjoys some hyperbolicity, with expansion constant is $\Lambda > 1$. See~\cite[\S4.4]{chernov2006chaotic} for a precise definition of the hyperbolicity of $T$ and the exact definition of $\Lambda$.

\begin{proposition}\label{lemma: regular points}
There exists a constant $C >0$ such that for every $0< \ve < \ve_0$, every $\delta >0 $, every $n \geqslant 2$,
\begin{align*}
\log r(n,\delta, B(x,\ve,n)) \leqslant 2 \frac{\log \# \cP}{\log \Lambda} \log \delta^{-1} + C, \quad \forall x \in M \smallsetminus \cS_{n-1}.
\end{align*} 
\end{proposition}
In particular, taking the appropriate limits in $n$, $\delta$ and $\ve$, notice that points of $M \smallsetminus \bigcup_{n \geqslant 0} \cS_n$ don't contribute to $h^*(T)$.

\begin{proof}
Let $x \in M \smallsetminus \cS_{n-1}$, $\ve < \ve_0$ and $\delta > 0$. Let $\cM_0^{n-1}(x)$ denote the unique element of $\cM_0^{n-1}$ containing $x$. Then, $B(x,\ve,n) \subset \cM_0^{n-1}(x)$. 

Note that if $F' \subset B(x,\ve,n)$ is $(\delta/2, n)$-separated of maximal cardinality, then $F'$ is $(\delta,n)$-spanning $B(x,\ve,n)$. In particular, $r(\delta,n,B(x,\ve,n)) \leqslant \# F'$. Furthermore, if $F \subset \cM_0^{n-1}(x)$ is $(\delta/2,n)$-separated with maximal cardinality, then $F$ is $(\delta,n)$-spanning and $\# F' \leqslant \# F$. Let $F$ be such a set.

Since $F$ is $(\delta/2,n)$-separated, according to the proof of~\cite[Lemma~3.4]{BD2020MME}, each element of $\hP_{-k_\delta}^{n-1+k_\delta}$ contains at most one point of $F$, where $k_\delta$ is the smallest integer such that $C_1 \Lambda^{-k_\delta} < c_1 \delta/2$, where $c_1$ and $C_1$ are constants coming respectively from~\cite[Lemma~3.4]{BD2020MME} and~\cite[eq. (3.1)]{BD2020MME}. Since $\cM_0^{n-1}(x)$ contains at most $\# \cM_{-1}^0$ elements of $\cM_{-1}^{n-1}$ and, from~\cite[Lemma~3.3]{BD2020MME}, $\cM_{-1}^{n-1} = \hP_0^{n-2}$, we get that $\cM_0^{n-1}(x)$ contains at most $\# \cM_{-1}^0 \# \cP_{-k_\delta}^0 \# \cP_0^{k_\delta +1}$ elements of $\hP_{-k_\delta}^{n-1+k_\delta}$. Therefore
\begin{align*}
r(n,\delta, B(x,\ve,n)) \leqslant \# F \leqslant \# \cM_{-1}^0 (\# \cP)^{2k_\delta +1} \leqslant \# \cM_{-1}^0 \# \cP  \exp \left( - 2\log \# \cP \left\lceil \frac{\log c_1 C_1^{-1} \delta}{\log \Lambda} \right\rceil  \right).
\end{align*}
Taking the $\log$ finishes the proof.
\end{proof}

\begin{proposition}\label{lemma: singular points}
There exists a constant $C >0$ such that for every $0< \ve < \ve_0$ small enough, every $\delta > 0$ and every $n\geqslant n_0$,
\begin{align*}
\log r(n,\delta, B(x,\ve,n)) \leqslant 2 \frac{\log \# \cP}{\log \Lambda} \log \delta^{-1} + C + \left( 3+ 2\left\lfloor \frac{\max \tau}{\min \tau} \right\rfloor \right)(n+n_0)s_0 \log 2K, \quad \forall x \in \cS_{n-1},
\end{align*} 
where $K$ is such that $K_n \leqslant Kn$ for every $n$.
\end{proposition}
Actually, the estimate in Proposition~\ref{lemma: singular points} holds for every $x \in M$, but in lights of Proposition~\ref{lemma: regular points} we only state it for points of $\cS_{n-1}$.

To prove the above, we rely on a following result which allows to count the number of times a flow orbit passes near a scatterer without making a collision.

\begin{lemma}\label{sublemma:near_miss}
For every $n_0$ fixed, there exists $\tilde \ve = \tilde \ve(n_0)$ such that for every $\ve \in (0, \tilde \ve)$ and every $x \in M$, the curve $\{ \Phi_t (x) \mid t \in [0 , \sum_{i=0}^{n_0-1} \tau(T^ix)] \} \subset \Omega $ passes at distance less than $\ve$ to the scatterers without making a collision at most $\lfloor \max \tau / \min \tau \rfloor s_0 n_0$ times.
\end{lemma}

\begin{proof}[Proof of Lemma~\ref{sublemma:near_miss}]
Denote $t_0 \coloneqq \lfloor n_0 \max \tau / \min \tau \rfloor $. First, note that $(x,t) \mapsto \Phi_t(x)$ is continuous on the compact set $\Omega \times [ - t_0 \max \tau , t_0 \max \tau ]$ and, therefore, is uniformly continuous on this same set. The idea of the proof is to find for every $x$ a point $y$ in a neighbourhood such that the flow orbit of $y$ makes a grazing collision on the scatterers getting close to the trajectory of $x$.  We start by considering the points making multiple grazing collisions. Let
\begin{align*}
R_2 \coloneqq \{ x \in M \mid (T^i x)_{0 \leqslant i \leqslant t_0} \text{ comprises at least two grazing collisions} \}. 
\end{align*}
Remark that each point of $R_2$ must be a corner of $\bar A$, for some $A \in \cM_0^{t_0}$. In particular $R_2$ is a finite set. Informally, define $\ve_2$ to be the minimal distance (in $\Omega$) between the scatterers and the segment $[ T^i x, T^{i+1}x ] \subset \Omega $, $x \in R_2$ and $0 \leqslant i < t_0 $. More precisely, define $\pi : \Omega \to Q$ by $\pi (x) = \pi (\bar x,v)= \bar x$ where $x=(\bar x,v) \in \Omega = Q \times \mathbbm{S}^1 / \sim$, and
\begin{align*}
\tilde d([x,Tx], \partial Q) \coloneqq \inf \{ \delta > 0 \mid \partial Q \cap (\pi ([x,Tx]) \times [-\delta,\delta]) \text{ has $3$ or more connected components} \}
\end{align*}
Therefore, let 
\begin{align*}
\ve_2 \coloneqq \min_{x \in R_2} \min_{0 \leqslant i < t_0} \tilde d([T^i x,T^{i+1} x], \partial Q) >0,
\end{align*}
and let $\delta_2$ be the delta associate to $\ve_2/2$ by the uniform continuity of the flow.

We now consider the points making a single grazing collision: let
\begin{align*}
R_1 \coloneqq \{ x \in M \mid (T^i x)_{0 \leqslant i \leqslant t_0} \text{ comprises exactly one grazing collisions} \} = \cS_{t_0} \smallsetminus R_2. 
\end{align*}
Then $R_1' \coloneqq \cS_{t_0} \smallsetminus \cN_{\delta_2/2}(R_2)$ is compact. Let
\begin{align*}
\ve_1 \coloneqq \min_{x \in R_1'} \min_{0 \leqslant i < t_0} \tilde d([T^i x,T^{i+1} x], \partial Q).
\end{align*}
By the uniform continuity of the flow, it follows that $\ve_1 > 0$. Let $\delta_1> 0$ be the delta associated with $\min( \ve_2/2, \ve_1/2, \delta_2/2, \diam R_2 /4)$ by the uniform continuity of the flow. Up to decreasing the value of $\delta_1$, we can assume that $\delta_1 < \min(\ve_1 /2,\ve_2/2)$. Finally, let $\delta_0$ be the delta associated to $\delta_1$ by the uniform continuity of the flow. We can assume that $\delta_0 < \delta_1$.

We claim that any $\ve > 0$ such that $\ve < \delta_0$ has the desired property. Indeed, assume that for some fixed $x \in M$ there exist two or more distinct $0 \leqslant i_j < n_0$ such that $\tilde d([T^{i_j} x,T^{i_j+1} x], \partial Q) < \ve$. Therefore, to each $i_j$ is associated a point $x_j \in M$, in a $\delta_1$-neighbourhood of $x$, such that the flow trajectories of $x_j$ and $x$ stay at distance at most $\min( \ve_2/2, \ve_1/2, \delta_2/2, \diam R_2 /4)$ for every time $t \in [0, n_0 \max \tau]$. By this shadowing, we get that the orbit of $x_j$ passes at distance less than $\delta_0 + \ve_1/2 < \ve_1$ to $\partial Q$ (corresponding to the in-between collisions associated with $i_k$, $k \neq j$). In particular, by the definition of $\ve_1$, we get that $x_j \notin R_1'$. Thus $d_M(x_j,R_2) < \delta_2/2$ for every $j$. Since $d_M(x,x_j) < \delta_1 \leqslant \diam R_2 /4 $, there exists a unique $y \in R_2$ such that $d_M(x_j,y) < \delta_2 /2$. Therefore $d(x,y) < \delta_2$. Thus, the flow orbits of $x$ and $y$ stay at distance less than $\ve_2/2$ to each other for times $t \in [0 , \sum_{i=0}^{n_0-1} \tau(T^ix)]$. Thus, the flow orbit of $y$ passes at distance less than $\ve + \ve_2 /2 < \ve_2$ to $\partial Q$ when the orbit of $x$ passes at distance less than $\ve$. Therefore, by the definition of $\ve_2$, the flow orbit of $y$ makes a grazing collision at this instants. However, the flow orbit of $y$, for $t \in [0, n_0 \max \tau]$ makes at most $\lfloor \max \tau / \min \tau \rfloor n_0$ collisions, and thus at most $\lfloor \max \tau / \min \tau \rfloor s_0 n_0$ grazing collisions. Taking $\tilde \ve = \delta_0$ finishes the proof.
\end{proof}

We have now all the tools in order to prove Proposition~\ref{lemma: singular points}.
\begin{proof}[Proof of Proposition~\ref{lemma: singular points}]
Let $0 < \ve < \ve_0$ be small enough so that $\sin \ve \leqslant \tilde \ve / \max \tau$ (where $\tilde \ve$ comes from Lemma~\ref{sublemma:near_miss}),
\begin{align*}
A \cap \cN_\ve(A \cap (\cS_{-1} \cup \cS_1)) \subset \bigcup_{-1 \leqslant i \leqslant 1} T^{i} \{ |\varphi| > \varphi_0 \} , \quad \forall A \in \cP,
\end{align*}
and $B(x,\ve)$ intersects at most $2K$ elements of $\cM_{-1}^1$ for every $x \in M$.

Let $x \in M$. Then $B(x,\ve,n)$ might intersects several elements of $\cM_{-1}^n = \hP_0^{n-1}$. We start by estimate the number of those sets. 

Assume first that $n=n_0$. For any $A \in \cP_0^{n-1}$, there exist $A_i \in \cP$ such that $A = \bigcap_{i=0}^{n-1} T^{-i} A_i$. Thus $A \cap B(x,\ve,n) \neq \emptyset$ if and only if $A_i \cap B(T^{-i}x,\ve) \neq \emptyset$ for every $0 \leqslant i < n$. In order to count the number of admissible $A_i$, we distinguish three mutually exclusive cases:

\noindent
i) $B(y,\ve) \cap (\cS_{-1} \cup \cS_1) = \emptyset$. Therefore, $B(y,\ve)$ is contained is a single element of $\cP$.

\noindent
ii) $y \in A \cap \cN_\ve(A \cap (\cS_{-1} \cup \cS_1))$ for some $A \in \cP$. Therefore at least one of $y$, $T y$ or $T^{-1}y$ is a $\varphi_0$-grazing collision, and $B(y,\ve)$ intersects at most $2K$ elements of $\cP$. This case occurs at most $3 s_0 n_0$ times.

\noindent
iii) Otherwise, at least one of $[y,Ty]$ or $[T^{-1}y,y]$ passes $\ve$ close to the scatterers without making a collision. Furthermore, $B(y,\ve)$ intersects at most $2K$ elements of $\cP$. By the Sublemma~\ref{sublemma:near_miss}, this case occurs at most $2 \lfloor \max \tau / \min \tau \rfloor s_0 n_0$ times.

Combining the three above situations, we obtain  
\begin{align*}
\#\{ A \in \cP_0^{n_0-1} \mid A \cap B(x,\ve,n_0) \neq \emptyset \} \leqslant (2K)^{\left( 3+ 2\left\lfloor \frac{\max \tau}{\min \tau} \right\rfloor \right) s_0 n_0}.
\end{align*}

Now, for general $n \geqslant n_0$, write $n=kn_0 +\ell$, $0 \leqslant \ell < n_0$. Then, any $A \in \cP_0^{n-1}$ can be written $A = \bigcap_{i=0}^{k-1} T^{in_0}A_i \cap B$, where $A_i \in \cP_0^{n_0-1}$ and $B \in \cP_0^{\ell -1}$. Thus, $A \cap B(x,\ve,n) \neq \emptyset$ if and only if $A_i \cap B(T^{in_0}x,\ve,n_0) \neq \emptyset$ for all $i$ and $B \cap B(T^{kn_0}x,\ve,\ell) \neq \emptyset$. Therefore, by the computation done in the case $n=n_0$, and since $\cP_0^{n-1} = \cM_{-1}^n$, we obtain
\begin{align*}
\#\{ A \in \cM_0^{n} \mid A \cap B(x,\ve,n) \neq \emptyset \} \leqslant \# \cM_{-1}^0 \left( (2K)^{\left( 3+ 2\left\lfloor \frac{\max \tau}{\min \tau} \right\rfloor \right) s_0 n_0} \right)^{\left\lceil \frac{n}{n_0} \right\rceil }.
\end{align*}
For each $A \in \cM_0^{n-1}$, let $F_A \subset A$ be $(\delta/2,n)$-separated with maximal cardinality if $A \cap B(x,\ve,n) \neq \emptyset$, and $F_A= \emptyset$ otherwise. Thus, the union of the $F_A$ is $(\delta,n)$-spanning $B(x,\ve,n)$. Therefore, using that 
\begin{align*}
\#\{ A \in \cM_0^{n-1} \mid A \cap B(x,\ve,n) \neq \emptyset \} \leqslant \#\{ A \in \cM_0^{n} \mid A \cap B(x,\ve,n) \neq \emptyset \},
\end{align*}
and estimates on each $\# F_A$ obtained in the proof of Proposition~\ref{lemma: regular points} yield
\begin{align*}
r(n,\delta,B(x, & \ve,n)) \leqslant \sum_{A \in \cM_0^{n-1}} \# F_A \\
&\leqslant (\# \cM_{-1}^0)^2 \# \cP  \exp \left( - 2\log \# \cP \left\lceil \frac{\log c_1 C_1^{-1} \delta}{\log \Lambda} \right\rceil  \right) \left( (2K)^{\left( 3+ 2\left\lfloor \frac{\max \tau}{\min \tau} \right\rfloor \right) s_0 n_0} \right)^{\left\lceil \frac{n}{n_0} \right\rceil }.
\end{align*}
Taking the $\log$ of the above estimate ends the proof.
\end{proof}

\begin{corollary}\label{prop: top tail entropy bound}
For any finite horizon billiard table, $h^*(T) \leqslant \left( 3+ 2\left\lfloor \frac{\max \tau}{\min \tau} \right\rfloor \right) s_0 \log 2 K$.
\end{corollary}

\begin{proof}
By Propositions~\ref{lemma: regular points} and \ref{lemma: singular points}, we get that 
\begin{align*}
\limsup_{n \to \infty} \frac{1}{n} \log \sup_{x \in M} r(n,\delta,B(x,\ve,n)) \leqslant \left( 3+ 2\left\lfloor \frac{\max \tau}{\min \tau} \right\rfloor \right) s_0 \log 2 K.
\end{align*}
Taking the appropriate limits in $\delta$ and in $\ve$ yields the result.
\end{proof}

\appendix
\section{Examples}\label{sect: examples}

Denote by $Q^{(d)}$ and $Q^{(R,R')}$ the billiard tables with parameters $d$ and $(R,R')$ respectively, as depicted in Figure~\ref{fig:tables}. These billiard tables have finite horizons and scatterers do not intersects whenever
\begin{align*}
d \in \, &\cD_1 \coloneqq (2, 4/\sqrt{3}), \\
(R,R') \in \, &\cD_2 \coloneqq \{ (r,r') \in \bR^2 \mid  0 < r < r' < \frac{1}{2} , \text{ and } \frac{1}{2} < r + r' < \frac{\sqrt{2}}{2} , \text{ and } r' > \frac{\sqrt{2}}{4} \}.
\end{align*}
Relying on numerical computations performed in respectively~\cite{Baras95} and~\cite{Garrido97}, Baladi and Demers proved that $Q^{(d)}$ and $Q^{(R,0.4)}$ satisfy the sparse recurrence condition for each of the finitely many values $d \in \cD_1$ and $(R,0.4) \in \cD_2$ for which the computations have been performed.

In this section, we prove that there is an open set of parameters $d$ (resp $R$) for which $Q^{(d)}$ (resp $Q^{(R,0.4)}$) satisfies the sparse recurrence condition. This is the content of Proposition\ref{prop: open sparse rec}. We also prove that for most of these parameters, the associated singularities are negligible. This is the content of Proposition~\ref{prop: Gdelta set of param}.

\begin{figure}[ht]
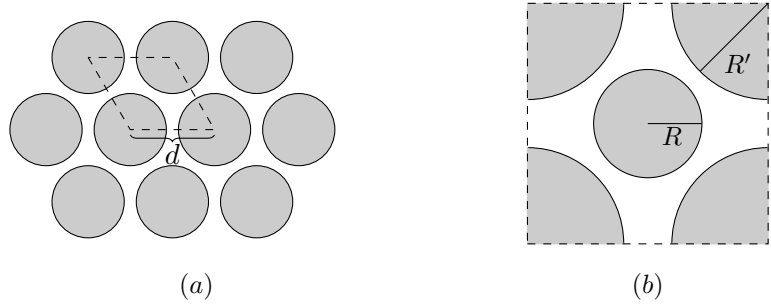

\tikz[x=8mm,y=8mm]
{
 \filldraw[fill=black!20!white, draw=black] (4.7,.7) circle (.6);
\filldraw[fill=black!20!white, draw=black] (6.1,.7) circle (.6);
\filldraw[fill=black!20!white, draw=black] (7.5,.7) circle (.6);
\filldraw[fill=black!20!white, draw=black] (4,1.9) circle (.6);
\filldraw[fill=black!20!white, draw=black] (5.4,1.9) circle (.6);
 \filldraw[fill=black!20!white, draw=black] (6.8,1.9) circle (.6);
  \filldraw[fill=black!20!white, draw=black] (8.2,1.9) circle (.6); 
\filldraw[fill=black!20!white, draw=black] (4.7,3.1) circle (.6);
\filldraw[fill=black!20!white, draw=black] (6.1,3.1) circle (.6);
\filldraw[fill=black!20!white, draw=black] (7.5,3.1) circle (.6);

\draw[dashed] (4.7,3.1) -- (5.4,1.9) -- (6.8,1.9) -- (6.1,3.1) -- cycle;

\draw[decoration={brace,mirror,raise=2pt},decorate]
  (5.4,1.9) -- node[below=1.6pt] {$d$} (6.8,1.9);

\node at (6.5,-.7){\small$(a)$};

\filldraw[fill=black!20!white, draw=black]  (13.6,0) arc (0:90:1.6);
\filldraw[fill=black!20!white, draw=black!20!white]  (13.6,0) -- (12,1.6) -- (12,0) -- cycle;
\filldraw[fill=black!20!white, draw=black]  (16,1.6) arc (90:180:1.6);
\filldraw[fill=black!20!white, draw=black!20!white]  (16,1.6) -- (14.4,0) -- (16,0) -- cycle;

\filldraw[fill=black!20!white, draw=black]  (14.4,4) arc (180:270:1.6);
\filldraw[fill=black!20!white, draw=black!20!white]  (14.4,4) -- (16,2.4) -- (16,4) -- cycle;

\filldraw[fill=black!20!white, draw=black]  (12,2.4) arc (270:360:1.6);
\filldraw[fill=black!20!white, draw=black!20!white]  (12,2.4) -- (13.6,4) -- (12,4) -- cycle;

\draw[dashed] (12,0) rectangle (16,4) ;

\filldraw[fill=black!20!white, draw=black]  (14,2) circle (.9);
\draw (14,2) -- (14.9,2);
\node at (14.4,1.8){\small $R$};

\draw (16,4) -- (14.87,2.87);
\node at (15.5, 3) {\small$R'$};

\node at (14,-.7){\small$(b)$}; 

}
\caption{(a) The Sinai billiard on a hexagonal lattice studied in~\cite{Baras95}, scatterer of radius 1, and distance $d$ between the centers of adjacent scatterers. (b) The Sinai billiard on a square lattice with scatterers of radius $R < R'$ studied in~\cite{Garrido97}. The boundary of a single cell is indicated by dashed lines in both tables. (Figure taken from~\cite[Figure~2]{BD2020MME}.)}
\label{fig:tables}
\end{figure}

\begin{proposition}\label{prop: open sparse rec}
For any $d \in (2.15 , 4/\sqrt{3})$ (resp. $R \in (0.1 , 0.25)$), the table $Q^{(d)}$ (resp. $Q^{(R,0.4)}$) has the sparse recurrence property with respect to the zero potential.
\end{proposition}

\begin{proof}
Since $\htop (T) \geqslant h_{\musrb}(T)$, It is sufficient to provide a lower bound on $h_{\musrb}(T)$ larger than $s_0 \log 2$. From~\cite[Equations (3.49),(3.50) and (3.33)]{chernov2006chaotic}, we get that 
\begin{align}\label{eq: lower bound srb}
h_{\musrb}(T) \geqslant \frac{1}{2} \int_{- \frac{\pi}{2}}^{\frac{\pi}{2}} \log \left( 1 + \tmin \frac{2 \kappa_{\min}}{\cos \varphi} \right) \cos \varphi \intd \varphi,
\end{align}
where $\kappa_{\min}$ denotes the minimal curvature of the scatterers.

For the family of billiard tables studied in~\cite{Baras95}, we get that $\kappa_{\min}=1$ and $\tmin = d -2$. By differentiating, it follows that the right-hand side of \eqref{eq: lower bound srb} is an increasing function of the variable $d$. Now, using numerical computations, we get that for $d_0=2.15$,
\begin{align*}
\frac{1}{2}\int_{- \frac{\pi}{2}}^{\frac{\pi}{2}} \log \left( 1 + \frac{2(d_0-2)}{\cos \varphi} \right) \cos \varphi \intd \varphi > 0.36 > \frac{1}{2}\log 2 \geqslant s_0 \log 2.
\end{align*}
Hence, for every $d \in (2.15, 4/\sqrt{3})$, $Q^{(d)}$ has the sparse recurrence property.

For the family of billiard table studied in~\cite{Garrido97}, we get that $\kappa_{\min} = \frac{1}{R'} = 2.5$ and
\begin{align*}
\tmin = \begin{cases}
1-2R' &\quad \text{ if  } 0.1 < R \leqslant \sqrt{2}/2 - (1-2R'), \\
\sqrt{2}/2 - (R+R') &\quad \text{ if  }  \sqrt{2}/2 - (1-2R') \leqslant R < \sqrt{2}/2 - R'.
\end{cases}
\end{align*}
Fixing $R'=0.4$, we get that the right-hand side of \eqref{eq: lower bound srb} is a decreasing function of the variable $R$. Now, using numerical computations, we get that for $R_0=1/4$
\begin{align*}
\frac{1}{2}\int_{- \frac{\pi}{2}}^{\frac{\pi}{2}} \log \left( 1 + \frac{5(\sqrt{2}/2 - R' - R_0 )}{\cos \varphi} \right) \cos \varphi \intd \varphi > 0.347 > \frac{1}{2}\log 2 \geqslant s_0 \log 2.
\end{align*}
Hence, for every $R \in (0.1 , 1/4)$, $Q^{(R,0.4)}$ has the sparse recurrence property.
\end{proof}

\begin{proposition}\label{prop: Gdelta set of param}
There exists a dense $G_\delta$ set $G_1 \subset \cD_1$ (resp. $G_2 \subset (0.1, \frac{\sqrt{2}}{2}-0.4)$) such that if $d \in G$ (resp. $R \in G_2$), then $Q^{(d)}$ (resp. $Q^{(R,0.4)}$) has negligible singularities.
\end{proposition}

\begin{proof}
The proof is very similar to the one of Proposition~\ref{prop: U_k are open and dense}. For completeness, we provide the modifications to be done. Denote
\begin{align*}
U_k \coloneqq \{ d \in \cD_1 \mid \bigcup_{i=3}^{k} \Fix \, T_d^i \cap \cS_0 = \emptyset \}, \, k \geqslant 3 ,
\end{align*}
where $T_d$ is the collision map associated to $Q^{(d)}$. Note that for $d_1, d_2 \in \cD_1$, 
\begin{align}\label{eq: equiv metric}
d_{C^3}(Q^{(d_1)},Q^{(d_2)}) = |d_1 - d_2|.
\end{align}
To prove that $U_k$ is open, let $d_0 \in U_k$ and $Q_0 = Q^{(d_0)}$. By contradiction, assume that for every $n \geqslant 1$, there exists $d_n \in \cD_1$ such that $|d_0 - d_n| < 1/n$. Denote $Q_n = Q^{(d_n)}$. The rest of the proof of openness is exactly as in the proof of Proposition~\ref{prop: U_k are open and dense}.

To prove that $U_k$ is dense in $\cD_1$, let $d_0 \in \cD_1$ and $\delta > 0$. The proof is then as in the one of Proposition~\ref{prop: U_k are open and dense} where the successive perturbed tables $Q_n = Q^{(d_n)}$ are obtained by constructing an increasing sequence $(d_n)_{n\geqslant 1}$ bounded from above by $d_0 + \delta$. The construction of this sequence uses the same argument as if all of the orbit of $x_n$ were composed of points of the type $y_i$ (and none of the type $z_i$), as well as the equivalence of metrics \eqref{eq: equiv metric}.

Finally, set $G_1 = \bigcap_{k \geqslant 3} U_k$. The construction of $G_2$ is analogous (where instead of increasing values of the parameter $d$, we use decreasing values of $R$).
\end{proof}

\bibliography{biblio}{}
\bibliographystyle{abbrv}

\end{document}